\documentclass[11pt]{amsart}
\usepackage{hyperref}
\usepackage{amsmath}
\usepackage{amssymb}
\usepackage{xypic}
\usepackage{yhmath}
\usepackage{setspace}
\usepackage{geometry}
\def\beqnn{\begin{eqnarray*}}\def\eeqnn{\end{eqnarray*}}

\newtheorem{theorem}{Theorem}[section]
\newtheorem{lemma}[theorem]{Lemma}
\newtheorem{proposition}[theorem]{Proposition}
\newtheorem{corollary}[theorem]{Corollary}

\theoremstyle{remark}
\newtheorem{remark}[theorem]{Remark}
\theoremstyle{claim}
\newtheorem{claim}[theorem]{Claim}

\theoremstyle{definition}
\newtheorem{definition}[theorem]{Definition}

\theoremstyle{problem}

\theoremstyle{conjecture}
\newtheorem{conjecture}[theorem]{Conjecture}
\theoremstyle{question}
\newtheorem{question}[theorem]{Question}

\numberwithin{equation}{section}

\begin{document}

\begin{center}
\title[Integral means spectrum functionals on Teichm\"uller spaces]{Integral means spectrum functionals on Teichm\"uller spaces}
\end{center}

\author{Jianjun Jin}
\address{School of Mathematics, Hefei University of Technology, Xuancheng Campus, Xuancheng 242000, P.R.China}
\email{jin@hfut.edu.cn, jinjjhb@163.com}
\thanks{The author was supported by National Natural Science Foundation of China (Grant Nos. 11501157)}




\begin{abstract}
In this paper we introduce and study the integral means spectrum (IMS) functionals on Teichm\"uller spaces. We show that  the IMS functionals on the closure of the universal Teichm\"uller space and the universal asymptotic Teichm\"uller space are both continuous.  During the proof, we consider the Pre-Schwarzian derivative model of universal asymptotic Teichm\"uller space and establish some new results for it. We also show that the integral means spectrum of any univalent function admitting a quasiconformal extension to the extended complex plane is strictly less than the universal integral means spectrum. 
\end{abstract}

\maketitle
\section{{\bf {Introduction}}}
Let $\Delta=\{z:|z|<1\}$ denote the unit disk in the
complex plane $\mathbb{C}$.  We denote the extended complex plane by $\widehat{\mathbb{C}}=\mathbb{C}\cup\{\infty\}$. For two positive numbers $A, B$, we write $A \asymp B$ if there are two constants $C_1>0, C_2>0$ which are independent of the arguments and such that $C_1 B\leq A\leq C_2 B$.

We denote the class of all univalent functions $f$ in $\Delta$ by $\mathcal{U}$.  Let $\mathcal{S}$ be the class of all univalent functions $f$ in $\Delta$ with $f(0)=f'(0)-1=0$.   We let $\mathcal{S}_b$ be the subclass of $\mathcal{S}$, which consists of all bounded univalent functions.

Let $t\in \mathbb{R}$. The {\em integral means spectrum} $\beta_f(t)$ for $f\in \mathcal{U}$ is defined as
\begin{equation}
\beta_f(t)=\limsup_{r \rightarrow 1^{-}}\frac{\log \int_{0}^{2\pi} |f'(re^{i\theta})|^t d\theta}{|\log(1-r)|}.\nonumber
\end{equation}
The famous Koebe function $\kappa$ is defined as
$$\kappa(z):=\frac{z}{(1-z)^2},\, z\in \Delta.$$
It is well known, see for example \cite{Po-0}, that, as $r\rightarrow 1^{-}$,
\begin{equation}\label{est--0}
\int_{0}^{2\pi} \frac{d\theta}{|1-re^{i\theta}|^{\gamma}}\asymp
\begin{cases}
\frac{1}{(1-r)^{\gamma-1}}, \; \text{if} \; \gamma >1, \\
\log \frac{1}{1-r}, \; \;\;  \text{if} \; \gamma=1,\\
1, \; \quad\,\,\,\,\quad\;\; \text{if} \; \gamma<1.
\end{cases}
\end{equation}
Then we see from (\ref{est--0}) that
\begin{equation}\beta{\kappa}(t)=
\begin{cases}
3t-1, \; \text{if} \; t \geq \frac{1}{3}, \\
0, \; \quad\;\;\;\;  \text{if} \; t\in [-1, \frac{1}{3}),\\
|t|-1, \; \text{if} \; t<-1.
\end{cases} \nonumber
\end{equation}

The {\em universal integral mean spectrum} $B(t)$ and $B_b(t)$ are defined as
\begin{equation}
B(t)=\sup\limits_{f\in \mathcal{S}} \beta_f(t),\,\,B_b(t)=\sup\limits_{f\in \mathcal{S}_b} \beta_f(t).\nonumber
\end{equation}
It is an important problem in the theory of univalent functions(conformal mappings) to determine the exact values of the universal integral mean spectra $B(t)$ and $B_b(t)$. We review some known results and open problems on this topic. It was proved by Makarov in \cite{M} that
\begin{theorem}$$
B(t)=\max \{B_b(t), 3t-1\},\, t\in \mathbb{R}.$$
\end{theorem}
\begin{remark}
It is easy to see from  Makarov's result that $B(t)=B_b(t)$ for $t\leq \frac{1}{3}$.\end{remark}

For large $t$, Feng and MacGregor showed in \cite{FM}  that
\begin{theorem}\label{th-fm}$$
B(t)=3t-1,\,\, t\geq \frac{2}{5}.$$
\end{theorem}
\begin{definition}We say a function $f\in \mathcal{S}$ is an {\em extremal function} for $B_b(t)$(or $B(t)$) if $\beta_f(t)=B_b(t)$ (or $\beta_f(t)=B(t)$).\end{definition}
\begin{remark} It should be pointed out that the extremal function for $B_b(t)$ is not asked to be contained in $\mathcal{S}_b$ in our definition.\end{remark}
\begin{remark}
Theorem \ref{th-fm} tells us that the Koebe function is an extremal function for $B(t)$ when $t\geq \frac{2}{5}$.
\end{remark}

Kayumov proved in \cite{Kay} that
\begin{theorem}
$$ B(t)>\frac{t^2}{5},\,\, 0<t\leq\frac{2}{5}.$$
\end{theorem}

Also, see \cite{Po-2},  we have\begin{theorem}$$
B_b(t)=t-1,\,\, t\geq 2.$$
\end{theorem}
\begin{remark}
When $t\geq2$, by using (\ref{est--0}), we can check that $-\log (1-z)$ is an extremal function for $B_b(t)$.
\end{remark}

Carleson and Makarov obtained in \cite{CM} that
\begin{theorem}\label{CM}
There is a constant $t_0<0$ such that
$$ B_b(t)=B(t)=|t|-1,\,\, t\leq t_0.$$
\end{theorem}
An interesting open problem is to find out the optimal $t_0$. It is only known that $t_0\leq -2$.
It is conjectured that $t_0=-2$, or equivalently that $B(-2)=1$, which is usually referred to as Brennan's conjecture. By the experimental work, Kraetzer has conjectured in \cite{K} that
\begin{conjecture}
\begin{equation}B_b(t)=\begin{cases}
\frac{t^2}{4}, \;\; \; \;\; \;  \;  {\textrm{if}} \;\, |t| \leq 2, \\
|t|-1, \; {\textrm{if}} \;\, |t|>2.
\end{cases}\nonumber
\end{equation}
\end{conjecture}
For more results on the integral means spectrum and related topics, see the monograph \cite{GM} and recent survey \cite{HSo} and references cited therein.

We let $\mathcal{S}_q$ be the class of all univalent functions $f$ that belong to $\mathcal{S}_b$ and admit a quasiconformal extension to $\widehat{\mathbb{C}}$. Later, if a univalent function $f$ belongs to $\mathcal{S}_q$,  we still use $f$ to denote its quasiconformal extension. Recently, some researchers have studied the integral means spectra of univalent functions admitting a quasiconformal extension to $\widehat{\mathbb{C}}$, see for example, \cite{HS-2}, \cite{Hed-1}, \cite{Hed-2}, \cite{Iv}, \cite{Jin}, \cite{PS} and \cite{Pr}. By the fractal approximation principle, see \cite{M}, \cite{CJ},  we know that
\begin{theorem}\label{quan}
For each $t\in \mathbb{R}$, we have
\begin{equation}\label{qua}
B_b(t)=\sup\limits_{f\in \mathcal{S}_q}\beta_f(t).
\end{equation}
\end{theorem}

We see from Theorem \ref{quan} that finding out the exact values of $B_b(t)$ can be thought of as a global extremal problem in the class $\mathcal{S}_q$. In this paper, we do not seek to find a better value of $B_b(t)$ but instead focus on the following  
\begin{question}\label{ques}
{\bf (1)} For fixed $t \in \mathbb{R}$ with $t\neq 0$, how does the integral means spectrum $\beta_f(t)$ depend on $f$ in the class $\mathcal{S}_q$? {\bf (2)} For each $t\neq 0$, does there exist at least one extremal function for $B_b(t)$? {\bf (3)} If the extremal functions for $B_b(t)$ exist, which subset of $\mathcal{S}$ do they belong to?  \end{question}
To answer this question, we introduce and study some functionals, induced by the integral means spectrum $\beta_f(t)$,  on the universal Teichm\"uller space $T$ and the universal asymptotic Teichm\"uller space $AT$. We call them integral means spectrum (IMS) functionals.  We will show that

\begin{theorem}\label{m-1}
For each $t\in \mathbb{R}$,  the IMS functional $$I_{{T}}: [\mu]_T  \mapsto  \beta_{f_{\mu}}(t), \,\, [\mu]_T \in {T},$$
is continuous.
\end{theorem}

\begin{theorem}\label{m-2}
For each $t\in \mathbb{R}$, the IMS functional $$I_{{AT}}: [\mu]_{AT}  \mapsto  \beta_{f_{\mu}}(t), \,\, [\mu]_{AT} \in {AT},$$
is continuous.
\end{theorem}

\begin{theorem}\label{m-3}
For each $t\in \mathbb{R}$, the IMS functional $$I_{\overline{T}_1}: \phi  \mapsto \beta_{f_{\phi}}(t), \,\, \phi \in \overline{T}_1,$$
is continuous.
\end{theorem}

The paper is organized as follows. We recall some basic definitions and results on univalent functions and quasiconformal mappings, and then prove some auxiliary results in the next section.  In Section 3,  after introducing some definitions and notations, we will restate Theorem \ref{m-1}, \ref{m-2}, \ref{m-3}. We shall present the proof of  Theorem \ref{m-1}, \ref{m-2}, \ref{m-3} in Section 4. In Section 5, we establish a final main theorem, which shows that the integral means spectrum of any univalent function admitting a quasiconformal extension  to $\widehat{\mathbb{C}}$ is strictly less than the universal integral means spectrum. We also give some remarks and questions in this last section.

\section{\bf {Preliminaries and auxiliary results}}
In this section, we first recall some basic results of univalent functions(conformal mappings) and quasiconformal mappings, for references,  see \cite{Du, Po, L, LV}. Then we establish some auxiliary results which will be needed later.

Let $\Delta^{*}=\widehat{\mathbb{C}}-\overline{\Delta}$ be the
exterior of $\Delta$ and
$\mathbb{T}=\partial\Delta=\partial\Delta^{*}$ be the unit circle. We use the notation $\Delta(z,r)$ to denote the disk centered at $z$ with radius $r$ and we will write $\Delta(r)$ to denote the disk centered at $0$ with radius $r$. We will use ${\textup{Area}}(\Omega)$ to denote the two dimensional Lebesgue measure of the measurable set $\Omega$ of the complex plane. For two compact sets $A, B$ of the complex plane, we define the distance of  $A$ and $B$, denoted by ${\text{dist}}(A, B)$, as
\begin{equation}
{\textup{dist}}(A, B):=\min_{x\in A, y\in B}|x-y|. \nonumber
\end{equation}
Let $\Omega$ be a simply connected proper subdomain of $\mathbb{C}$.  We shall use $\rho_{\Omega}$ to denote the hyperbolic metric with curvature $-4$ in $\Omega$.  That is 
$$\rho_{\Omega}(w)=\frac{|\tau'(w)|}{1-|\tau(w)|^2},\,\, w\in \Omega.$$
Here $\tau$ is a univalent function from $\Omega$ to $\Delta$.  In particular, $$\rho_{\Delta}(z)=(1-|z|^2)^{-1}, z\in \Delta.$$ 
For an analytic function $f$ in the simply connected domain $\Omega$ with $f(w)\neq 0$ for all $w\in \Omega$, we will use $\log f$ to denote a certain single-valued branch of the logarithm of $f$. We will use $C(\cdot), C_{1}(\cdot), C_{2}(\cdot),\cdots$ to denote some positive numbers which depend only on the elements in the bracket and the numbers may be different in different places. 

\subsection{Univalent functions (conformal mappings)} First, we have
\begin{proposition}\label{pro-1}
Let $f\in \mathcal{U}$. Then, for any $z\in \Delta$,
\begin{equation} \frac{1}{4}(1-|z|^2)|f'(z)|\leq  {\textup{dist}}(f(z), \partial f(\Delta))\leq (1-|z|^2)|f'(z)|. \nonumber \end{equation}
\end{proposition}
\begin{remark}\label{re-1}
Let $f$ be a univalent function in $\Delta$ and $g$ be a univalent function in $f(\Delta)$. It follows from Proposition \ref{pro-1} that
\begin{equation}
{\textup{dist}}(g\circ f(z), \partial g(f(\Delta))) \leq (1-|z|^2)|[g\circ f(z)]'|= (1-|z|^2)|g'\circ f(z)|\cdot |f'(z)|, \nonumber
\end{equation}
and
\begin{equation}
(1-|z|^2)|f'(z)|\leq 4{\textup{dist}}(f(z), \partial f(\Delta)).\nonumber \end{equation}
Consequently, we have
$${\textup{dist}}(g(\zeta), \partial g(f(\Delta))) \leq 4 |g'(\zeta)|{\textup{dist}}(\zeta, \partial f(\Delta)).$$
Here $z\in \Delta, \zeta=f(z)\in f(\Delta).$
\end{remark}

 \begin{remark}\label{re-2}
Let $f$ be a univalent function from $\Delta$ to a bounded Jordan domain $\Omega$ in $\mathbb{C}$. Then we know that  $f$ can be extended to a mapping which is a homeomorphism from $\overline{\Delta}$ to $\overline{\Omega}$.
The extended mapping is still denoted by $f$. Then we have
\begin{equation}
{\textup{dist}}(f(z), f(\mathbb{T}))={\textup{dist}}(0, \mathbf{K}_f(\mathbb{T}))(1-|z|^2)|f'(z)|,\,\,z\in \Delta. \nonumber
\end{equation}
Here $\mathbf{K}_f$ is the Koebe transform of $f$ (with respect to $z\in \Delta$), defined as
\begin{equation}
\mathbf{K}_f(w):=\frac{f(\sigma_{z}(w))-f(z)}{(1-|z|^2)|f'(z)|},\,\,\,\sigma_{z}(w)=\frac{z+w}{1+\overline{z}w},\,\,\, w\in \Delta.\nonumber \end{equation}
On the other hand, by Proposition \ref{pro-1}, we have
$${\textup{dist}}(0, \mathbf{K}_f(\mathbb{T}))\asymp |\mathbf{K}_f(e^{i\arg z})-\mathbf{K}_f(0)|=\frac{|f(z)-f(e^{i\arg z})|}{(1-|z|^2)|f'(z)|}.$$
Then we obtain that
\begin{equation}
{\textup{dist}}(f(z), f(\mathbb{T}))\asymp |f(z)-f(e^{i\arg z})|,\, z\in \Delta. \nonumber
\end{equation}
\end{remark}

Let $\mathcal{A}(\Delta)$ denote the class of all analytic functions in $\Delta$.    We let ${E}_j$ be  the Banach space of functions $\phi \in \mathcal{A}(\Delta)$ with the norm
  \begin{equation}
 \|\phi\|_{E_j}:= \sup_{z\in \Delta} |\phi(z)|(1-|z|^2)^{j}<\infty, \,\,\,\, j=1,2. \nonumber
 \end{equation}

Let $f$ be a locally univalent function in an open domain $\Omega$ of $\mathbb{C}$. The {\em Pre-Schwarzian derivative} $N_f$ of $f$, and the {\em Schwarzian derivative} $S_f$ of $f$ are defined as
 $$N_f(z):=\frac{f''(z)}{{f'(z)}}, \, z\in \Omega,$$ and 
 $$S_f(z):=[N_f(z)]'-\frac{1}{2}[N_f(z)]^2=\frac{f'''(z)}{f'(z)}-\frac{3}{2}\left[\frac{f''(z)}{f'(z)}\right]^2, \,z\in \Omega.$$
Let $g$ be another locally univalent function in $f(\Omega)$. Then we have
\begin{equation}\label{chain}N_{g \circ f}(z)=N_{g}(f(z))[f'(z)]+N_f(z),\,\, z\in \Omega,  \end{equation}
and \begin{equation}\label{chain-1}S_{g \circ f}(z)=S_{g}(f(z))[f'(z)]^2+S_f(z),\,\, z\in \Omega. \end{equation}

It is well known that
 $$|N_f(z)|(1-|z|^2)\leq 6,\,\,{\text {and}}\,\, |S_f(z)|(1-|z|^2)^2\leq 6,$$
 for all $f\in \mathcal{U}$.  This means that
  $\|N_f\|_{E_1}\leq 6\,\,{\text {and}}\,\, \|S_f\|_{E_2}\leq 6$ for any $f\in \mathcal{U}$.

 We define the classes $\mathbf{N}$ and $\mathbf{S}$ as
$$
 \mathbf{N}=\{\phi \in \mathcal{A}(\Delta): \, \phi=N_f(z), f\in \mathcal{U}\}, \,\,\,
 \mathbf{S}=\{\phi \in \mathcal{A}(\Delta): \, \phi=S_f(z), f\in \mathcal{U}\}.
$$
Then we see that $\mathbf{N} \subset E_1$ and $\mathbf{S} \subset E_2$. Moreover, we have
\begin{proposition}\label{pro-2}
$\mathbf{N}$ and $\mathbf{S}$ are closed in $E_1$ and  $E_2$, respectively.
\end{proposition}
\begin{remark} It has been proved in \cite[Page 115]{L} that $\mathbf{S}$ is closed in $E_2$.  The statement that {\em $\mathbf{N}$ is closed in $E_1$} is also known in the literature. For the 
completeness of the paper, we will present a direct proof for this statement and some arguments of the proof will be used later.  First, we have
\begin{claim}\label{cla---}
For two locally univalent functions $f_1, f_2$ in $\Delta$, we have
$N_{f_1}(z)=N_{f_2}(z)$ for all $z\in \Delta$ if and only if there are two numbers $a, b\in \mathbb{C}$ such that $f_1(z)=af_2(z)+b$ for all $z\in \Delta$.
\end{claim}
The if part of this claim is trivial. We only prove the only if part. We let $N_{f_j}(z):=\phi(z), j=1,2$. Then
$$f_{j}(z):=\int_{0}^{z} e^{\int_{0}^{\zeta}\phi(w) dw+\log f_{j}'(0)}d\zeta+f_j(0),\, j=1,2.$$
It follows that
\begin{eqnarray}f_1(z) \nonumber &=&\int_{0}^{z} e^{\int_{0}^{\zeta}\phi(w) dw+\log f_{1}'(0)}d\zeta+f_1(0).\\&=&e^{[\log f'_1(0)-\log f'_2(0)]}[\int_{0}^{z} e^{\int_{0}^{\zeta}\phi(w) dw+\log f_{2}'(0)}d\zeta+f_2(0)]
\nonumber \\
&&\quad\quad\quad +f_1(0)-e^{[\log f'_1(0)-\log f'_2(0)]}f_2(0) .\nonumber
\end{eqnarray}
This proves the only if part of the claim and the proof of Claim \ref{cla---} is done.

We proceed to prove the statement.  We suppose that there is a sequence $\{{\bf f}_n\}_{n=1}^{\infty}$, ${\bf f}_n\in \mathcal{U}$ and ${\Phi}\in E_1$ such that $$\lim\limits_{n\rightarrow \infty}\|N_{{\bf f}_n}-\Phi\|_{E_1}=0.$$
First we have
$$\Phi(z)=\lim\limits_{n\rightarrow \infty}N_{{\bf f}_{n}}(z),$$
for each $z\in \Delta$. We will show that there is an $f_{\Phi}\in \mathcal{U}$ such that $N_{f_{\Phi}}=\Phi$. Without loss of generality,  since the Pre-Schwarzian derivative is affine invariant, we assume that ${\bf f}_n \in \mathcal{S}$ for all $n$.
Noting that $\{{\bf f}_n\}$ is a normal family, then $\{{\bf f}_n\}$ contains a subsequence (still denoted by $\{{\bf f}_n\}$) which is locally uniformly convergent in $\Delta$. We denote by $g$ the limit of the subsequence, i.e.,
$$g(z)=\lim\limits_{n\rightarrow \infty}{\bf f}_{n}(z), $$
and we know that $g\in \mathcal{S}$. Then at every point $z\in \Delta$ we have
$$N_{g}(z)=\lim\limits_{n\rightarrow \infty}N_{{\bf f}_{n}}(z).$$
On the other hand, for $\Phi\in E_1$, there is a locally univalent function $f_{\Phi}$ in $\Delta$ such that $$\Phi(z)=N_{f_{\Phi}}(z),\,\, z\in \Delta.$$
For example, we can take
$$f_{\Phi}(z):=\int_{0}^{z} e^{\int_{0}^{\zeta}\Phi(w) dw}d\zeta, \,\,\,z\in\Delta.$$
Consequently, we obtain that $N_g(z)=N_{f_{\Phi}}(z).$
Then we know from the above claim that $f_{\Phi}$ is univalent in $\Delta$.  This means that $\mathbf{N}$ is closed in $E_1$.
\end{remark}

\subsection{Quasiconformal mappings} We say a sense-preserving homeomorphism $f$,  from one open domain $\Omega$ in
$\mathbb{C}$ to another is a quasiconformal mapping if it has locally square
integral distributional derivatives and satisfies the Beltrami equation $\bar{\partial}f=\mu_f  \partial{f}$ with
$$\|\mu_f\|_{\infty}=\mathop{\text {ess sup}}\limits_{z\in \Omega} |\mu_f(z)|<1.$$
Here the function $\mu_f(z)$ is called the {\em  Beltrami coefficient} of $f$ and
$$\bar{\partial}f=f_{\bar{z}} :=\frac{1}{2}\left(\frac{\partial}{\partial x}+i\frac{\partial}{\partial y}\right)f,\,\,\, \partial f=f_{z} :=\frac{1}{2}\left(\frac{\partial}{\partial x}-i\frac{\partial}{\partial y}\right)f.$$

Let
$$D_f(z):=\frac{1+|\mu_f(z)|}{1-|\mu_f(z)|}, \,\,K_f:= \frac{1+\|\mu_f\|_{\infty}}{1-\|\mu_f\|_{\infty}}.$$
We call $D_f(z)$ the {\em dilatation function} of $f$ in $\Omega$ and $K_f$ is called the {\em dilatation} of $f$.

Let $f$ be a quasiconformal mapping from one open domain $\Omega_1$ to another domain $\Omega_2$. If $g$ is another quasiconformal mapping from $\Omega_1$ to $\Omega_3$. Then the Beltrami coefficients of $f, g$ and $g \circ f^{-1}$ satisfy the following chain rule.
\begin{equation}\label{dila}
\mu_{g \circ f^{-1}}\circ f(z) =\frac{1}{\tau}\frac{\mu_g(z)-\mu_f(z)}{1-\overline{\mu_f(z)}\mu_g(z)},\,\, \tau=\frac{\overline{\partial f}}{\partial f}, \,\,\, z\in \Omega_1.
\end{equation}
 Let $f$ be a bounded univalent function in a Jordan domain $\Omega$ of $\mathbb{C}$ admitting a quasiconformal mapping (still denoted by $f$) to $\widehat{\mathbb{C}}$.  The {\em boundary dilatation} of $f$, denoted by $b(f)$, is defined as
\begin{equation}
b(f):=\inf\{\|\mu_f|_{\Omega^{*}-E}\|_{\infty}: \, E {\text { is a compact set in}}\,\, \Omega^{*}\}. \nonumber
\end{equation}
Here $ \Omega^{*}=\widehat{\mathbb{C}}-\overline{\Omega}$ is seen as an open set in the Riemann sphere $\widehat{\mathbb{C}}$ under the spherical distance and $b(f)$ is the infimum of $\|\mu_f|_{\Omega^{*}-E}\|_{\infty}$ over all compact subsets $E$ contained in $\Omega^{*}$.

We say a domain $D$ in $\mathbb{C}$ is  a {\em ring domain} if it can be conformally mapped into an annulus $\{z: 0<r_1<|z|<r_2<\infty\}$, and the module ${\textup{Mod}}(D)$ of such ring domain $D$ is defined as ${\textup{Mod}}(D)=\log \frac{r_2}{r_1}$.
\begin{lemma}\label{l0}
Let $D_1$ be a ring domain in $\mathbb{C}$. Let $D_2$ be another ring domain such that ${D_2}$ is contained in $D_1$ and $D_1-{D_2}$ is not connected. Then we have
\begin{equation}
{\textup{Mod}}(D_2) \leq {\textup{Mod}}(D_1).\nonumber
\end{equation}
\end{lemma}

\begin{lemma}\label{l1}
Let $f$ be a quasiconformal mapping from one open domain $\Omega_1$ of $\mathbb{C}$ to another domain $\Omega_2$ in $\mathbb{C}$. Let $D$ be a ring domain with $\overline{D} \subset \Omega_1$. Then we have
\begin{equation}
{\textup{Mod}}(D)/K_f\leq {\textup{Mod}}(f(D))\leq K_f {\textup{Mod}}(D).\nonumber
\end{equation}
\end{lemma}

\begin{lemma}\label{l2}
Let $f$ be a quasiconformal mapping from $\widehat{\mathbb{C}}$ to itself with $f(\infty)=\infty$. Then, for any $r>0, z\in \mathbb{C}$, we have
\begin{equation}
\frac{\max\limits_{\theta\in [0, 2\pi)}|f(z+re^{i\theta})-f(z)|}{\min\limits_{\theta\in [0,2\pi)}|f(z+re^{i\theta})-f(z)|}\leq C(K_f).\nonumber
\end{equation}
\end{lemma}

Then we obtain that
\begin{lemma}\label{l3}
Let $f$ be a quasiconformal mapping from $\widehat{\mathbb{C}}$ to itself with $f(\infty)=\infty$.  Assume that $f$ maps one domain $\Omega_1$ of $\mathbb{C}$ to another domain $\Omega_2$ in $\mathbb{C}$. For $z\in \Omega_1$, let $0<r_1<r_2$ be such that $\Delta(z, r_2)\subset \Omega_1$.  Then we have
\begin{equation}\label{ll-3-0}
 \frac{\max\limits_{\theta\in [0,2\pi)}|f(z+r_2 e^{i\theta})-f(z)|}{\min\limits_{\theta\in [0,2\pi)}|f(z+r_1 e^{i\theta})-f(z)|}\leq C(K_f) \Big(\frac{r_2}{r_1}\Big)^{K_f}.
\end{equation}
\end{lemma}
\begin{proof}
We set
$$\mathbf{M}_2=\max\limits_{\theta\in [0,2\pi)}|f(z+r_2 e^{i\theta})-f(z)|, \,\, \mathbf{m}_2=\min\limits_{\theta\in [0,2\pi)}|f(z+r_2 e^{i\theta})-f(z)|,$$
$$\mathbf{M}_1=\max\limits_{\theta\in [0,2\pi)}|f(z+r_1 e^{i\theta})-f(z)|, \,\, \mathbf{m}_1=\min\limits_{\theta\in [0,2\pi)}|f(z+r_1 e^{i\theta})-f(z)|.$$
From Lemma \ref{l2}, we know that there is a positive constant $C_1(K_f)$ such that
 \begin{equation}\label{ll-3}
\mathbf{M}_1 \leq C_1(K_f) \mathbf{m}_1,\,\, \mathbf{M}_2 \leq C_1(K_f) \mathbf{m}_2.
\end{equation}
If $\mathbf{m}_2 \leq C_1(K_f)\mathbf{m}_1$, then it is easy to see from (\ref{ll-3}) that (\ref{ll-3-0}) holds in this case.
On the other hand, if $\mathbf{m}_2>C_1(K_f)\mathbf{m}_1\geq \mathbf{M}_1$, then we see that the annulus $\mathbf{A}:=\Delta(f(z), \mathbf{m}_2))-\Delta(f(z), \mathbf{M}_1)$ is contained in the ring domain $\mathbf{R}:=f(\Delta(z,r_2))-f(\Delta(z,r_1))$.  Consequently, we obtain from Lemma \ref{l0} that
$${\textup{Mod}}(\mathbf{A})\leq {\textup{Mod}}(\mathbf{R}).$$
It follows from (\ref{ll-3}) again and Lemma \ref{l1} that
\begin{eqnarray}\log\frac{\mathbf{M}_2}{\mathbf{m}_1}&\leq&  \log C_2(K_f)\cdot \frac{\mathbf{m}_2}{\mathbf{M}_1}=\log C_2(K_f)+{\textup{Mod}}(\mathbf{A})\nonumber \\
&\leq&\log C_2(K_f)+{\textup{Mod}}(\mathbf{R})\leq\log C_2(K_f)+K_f \log \frac{r_2}{r_1}.\nonumber \end{eqnarray}
Here $C_2(K_f)=[C_1(K_f)]^2$. Then (\ref{ll-3-0}) follows and the lemma is proved.
\end{proof}
We will need the following result due to Mori.
\begin{lemma}\label{l4}
Let $f$ be a quasiconformal mapping from $\Delta$ to itself with $f(0)=0$.  Then we have
$$|f(z_1)-f(z_2)|\leq 16 |z_1-z_2|^{\frac{1}{K_f}}, z_1, z_2\in \overline{\Delta}.$$
\end{lemma}

\begin{lemma}\label{l5}
Let $f$ belong to $\mathcal{S}_b$ and admit a quasiconformal extension to $\widehat{\mathbb{C}}$ with $f(\infty)=\infty$.  Then we have
$$C_1(K_f)(1-|z|^2)^{\|\mu_f\|_{\infty}}\leq  |f'(z)|\leq C_2(K_f) (1-|z|^2)^{-\|\mu_f\|_{\infty}}, \,\, z\in \Delta.$$
\end{lemma}

\begin{remark}\label{remark}
From Lemma \ref{l5}, for $z=|z|e^{i\arg z} \in \Delta$, we have
\begin{eqnarray}
|f(z)-f(e^{i \arg z})| &=& |\int_{z}^{e^{i \arg z}} f'(w)dw|  \nonumber \\
&=& |\int_{|z|}^1 f'(t e^{i \arg z})e^{i\arg z}dt| \nonumber  \\
& \leq & C_3(K_f) \int_{|z|}^1 (1-|t|^2)^{-\|\mu_f\|_{\infty}} dt \nonumber \\
&\leq & C_4(K_f) (1-|z|^2)^{\frac{2}{1+K_f}}. \nonumber
\end{eqnarray}
Here the first integral is taken on the radial path from $z$ to $e^{i\arg z}$.
\end{remark}

We next establish the following auxiliary result, which plays an important role in our later arguments and which also generalizes some related known results in \cite{BP} and \cite{Dyn}.
\begin{proposition}\label{key}
Let $f$ belong to $\mathcal{S}_b$ and admit a quasiconformal extension to $\widehat{\mathbb{C}}$ with $f(\infty)=\infty$. Let $\mathbf{h}$ be a bounded univalent function in $f(\Delta)$ with $\mathbf{h}(0)=\mathbf{h}'(0)-1=0$. We assume that $\mathbf{h}$ admits a quasiconformal extension to $\widehat{\mathbb{C}}$ with $\mathbf{h}(\infty)=\infty$ and the boundary dilatation $b(\mathbf{h})$ of $\mathbf{h}$ satisfies that  $3b(\mathbf{h})<1-\|\mu_f\|_{\infty}$.  Then, for any $0<\varepsilon <\frac{1}{3}(1-\|\mu_f\|_{\infty})-b(\mathbf{h})$, there is a constant $\delta>0$ such that
$$ |N_{\mathbf{h}}(\zeta)|{\textup{dist}}(\zeta, f(\mathbb{T}))<C(f, \mathbf{h})[ b(\mathbf{h})+\varepsilon],$$
for all  $\zeta\in f(\Delta)$ with ${\textup{dist}}(\zeta, f(\mathbb{T}))<\delta$.
\end{proposition}
\begin{remark}
In Proposition \ref{key}, the restricted condition $3b(\mathbf{h})<1-\|\mu_f\|_{\infty}$ may be not the best, but it is enough for this paper.
\end{remark}

\subsection{Proof of Proposition \ref{key}}
For any $0<\varepsilon <\frac{1}{3}(1-\|\mu_f\|_{\infty})-b(\mathbf{h})$,  in view of the definition of $b(\mathbf{h})$, we can find two numbers $R_1\in (0,1)$, $R_2\in (1,2)$ such that
\begin{equation}|\mu_{\mathbf{h}}(\zeta)|<b(\mathbf{h})+\varepsilon,\,\,\, a.e. \,\, \zeta \in f(\Delta(R_2))- f(\Delta),\nonumber \end{equation}
and
\begin{equation}
{\textup{dist}}(f(\mathbb{T}_1), f(\mathbb{T}))= {\textup{dist}}(f(\mathbb{T}_2), f(\mathbb{T})):=\mathbf{d}. \nonumber
\end{equation}
Here $\mathbb{T}_j=\partial \Delta(R_j), j=1,2.$

We take $\delta_1= \frac{\mathbf{d}}{2^{10}}.$  Then, for any $\zeta\in f(\Delta)$ with ${\textup{dist}}(\zeta, f(\mathbb{T}))<\delta_1$, there is a point $\zeta_0\in f(\mathbb{T})$ such that
${\textup{dist}}(\zeta, f(\mathbb{T}))=|\zeta_0-\zeta|:=\mathbf{d}_0.$ We let $\mathbf{m}\in \mathbb{N}$ be the biggest number such that
\begin{equation}
\Delta(\zeta, 2^{\mathbf{m}+1}\mathbf{d}_0) \subset  f(\Delta(R_2))-f(\Delta(R_1)). \nonumber
\end{equation}
It is easy to see that $2^{\mathbf{m}+1}\mathbf{d}_0 \geq  \frac{1}{2}\mathbf{d}$. By Pompeiu's formula, we have
\begin{equation}\label{pom}\mathbf{h}(\zeta)=\frac{1}{2\pi i}\oint_{\mathcal{C}} \frac{{\mathbf{h}}(w)}{w-\zeta}dw-\frac{1}{\pi}\iint_{\Delta(\zeta, \mathbf{r})}\frac{\bar{\partial} \mathbf{h}(w)}{w-\zeta}dudv.\end{equation}
Here $\mathcal{C}=\partial{\Delta}(\zeta, \mathbf{r})$ is a circle and $\mathbf{r}=2^{\mathbf{m}+1}\mathbf{d}_0\geq\frac{1}{2}\mathbf{d}$.

We take $\mathbf{M}=\max\limits_{\zeta\in f(\overline{\Delta(2)})}|\mathbf{h}(\zeta)|.$ Differentiating twice on the both sides of  (\ref{pom}), we get that
\begin{eqnarray}\label{key-1}|{\mathbf{h}}''(\zeta)| &\leq &  \frac{1}{\pi}\Big |\oint_{\mathcal{C}} \frac{\mathbf{h}(w)}{(w-\zeta)^3}dw\Big |+\frac{2}{\pi}\Big|\iint_{\Delta(\zeta, \mathbf{r})}\frac{\bar{\partial} \mathbf{h}(w)}{(w-\zeta)^3}dudv\Big| \\
&\leq & \frac{2\mathbf{M}}{{\mathbf{r}}^2} + \frac{2}{\pi}\iint_{\Delta(\zeta, \mathbf{r})-f(\Delta)}\frac{|\bar{\partial} \mathbf{h}(w)|}{|w-\zeta|^3}dudv \nonumber \\
&\leq & \frac{8\mathbf{M}}{{\mathbf{d}}^2}+\frac{2}{\pi}\iint_{\Delta(\zeta, \mathbf{r})-f(\Delta)}\frac{|\bar{\partial} \mathbf{h}(w)|}{|w-\zeta|^3}dudv. \nonumber
\end{eqnarray}
Note that
\begin{eqnarray}
\lefteqn{\iint_{\Delta(\zeta, \mathbf{r})-f(\Delta)}\frac{|\bar{\partial} \mathbf{h}(w)|}{|w-\zeta|^3}dudv} \nonumber \\&& = \sum_{k=0}^{{\mathbf{m}}}\iint_{\Delta(\zeta, 2^{k+1}\mathbf{d}_0)-\Delta(\zeta, 2^k \mathbf{d}_0)}\frac{|\bar{\partial}\mathbf{h}(w)|}{|w-\zeta|^3}dudv:=\sum_{k=0}^{{\mathbf{m}}} \mathbf{I}_k.  \nonumber
\end{eqnarray}
On the one hand, for any integral $k\in [0, {\mathbf{m}}]$, we have
\begin{eqnarray}
\lefteqn{\mathbf{I}_k=\iint_{\Delta(\zeta, 2^{k+1}\mathbf{d}_0)-\Delta(\zeta, 2^k \mathbf{d}_0)}\frac{|\bar{\partial} \mathbf{h}(w)|}{|w-\zeta|^3}dudv} \nonumber\\&& \leq \frac{1}{[2^k\mathbf{d}_0]^3}
\iint_{\Delta(\zeta, 2^{k+1}\mathbf{d}_0)}|\bar{\partial}  \mathbf{h}(w)|dudv \nonumber \\
&&  =\frac{1}{[2^k\mathbf{d}_0]^3}
\iint_{\Delta(\zeta, 2^{k+1}\mathbf{d}_0)} |{\partial} \mathbf{h}(w)|\cdot |\mu_{ \mathbf{h}}(w)|dudv \nonumber \\
&& =  \frac{1}{[2^k\mathbf{d}_0]^3}
\iint_{\Delta(\zeta, 2^{k+1}\mathbf{d}_0)} \left[\frac{J_{\mathbf{h}}(w)}{1-|\mu_{\mathbf{h}}(w)|^2}\right]^{\frac{1}{2}}\cdot |\mu_{ \mathbf{h}}(w)|dudv. \nonumber
\end{eqnarray}
Here $J_{\mathbf{h}}$ is the Jacobian of $\mathbf{h}$. It follows from Cauchy-Schwarz's inequality that
 \begin{eqnarray}\label{Ik}
\mathbf{I}_k &\leq&  \frac{1}{[2^k\mathbf{d}_0]^3}
 \Big(\iint_{\Delta(\zeta, 2^{k+1}\mathbf{d}_0)}\frac{|\mu_{ \mathbf{h}}(w)|^2}{1-|\mu_{\mathbf{h}}(w)|^2}dudv\Big)^{\frac{1}{2}}
 \\
&& \times \Big(\iint_{\Delta(\zeta, 2^{k+1}\mathbf{d}_0)} J_{\mathbf{h}}(w) dudv\Big)^{\frac{1}{2}} \nonumber  \\
 & \leq & \frac{b(\mathbf{h})+\varepsilon}{\sqrt{1-[b(\mathbf{h})+\varepsilon]^2}}\frac{2\sqrt{\pi}}{[2^k\mathbf{d}_0]^2}
  \Big[{\text {Area}}(\mathbf{h}(\Delta(\zeta, 2^{k+1}\mathbf{d}_0)))\Big]^{\frac{1}{2}}.\nonumber
\end{eqnarray}
Here and later, the notation ${\text{Area}}(E)$ denotes the Lebesgue measure of set $E$ in the complex plane.  On the other hand, by Lemma \ref{l2}, we have
\begin{equation}\label{ar-1}
[{\text {Area}}(\mathbf{h}(\Delta(\zeta, 2^{k+1}\mathbf{d}_0)))]^{\frac{1}{2}}\leq C_1(\mathbf{h}) \max\limits_{\theta\in [0, 2\pi)}|\mathbf{h}(\zeta+2^{k+1}\mathbf{d}_0e^{i\theta})-\mathbf{h}(\zeta)|,
\end{equation}
and
\begin{equation}\label{ar-2}
[{\text {Area}}(\mathbf{h}(\Delta(\zeta, \mathbf{d}_0)))]^{\frac{1}{2}}\geq C_2(\mathbf{h}) \min\limits_{\theta\in [0, 2\pi)}|\mathbf{h}(\zeta+\mathbf{d}_0e^{i\theta})-\mathbf{h}(\zeta)|.
\end{equation}
By Lemma \ref{l3}, we have
\begin{equation}\label{ar-3}
\frac{ \max\limits_{\theta\in [0, 2\pi)}|\mathbf{h}(\zeta+2^{k+1}\mathbf{d}_0e^{i\theta})-\mathbf{h}(\zeta)|}{\min\limits_{\theta\in [0, 2\pi)}|\mathbf{h}(\zeta+\mathbf{d}_0e^{i\theta})-\mathbf{h}(\zeta)|}
\leq C_3(\mathbf{h})(2^{k+1})^{\frac{1+b(\mathbf{h})+\varepsilon}{1-b(\mathbf{h})-\varepsilon}}.
\end{equation}
Then, combining (\ref{ar-1})-(\ref{ar-3}), we obtain that
\begin{equation}
\Big[\frac{{\text {Area}}(\mathbf{h}(\Delta(\zeta, 2^{k+1}\mathbf{d}_0)))}
{{\text {Area}}(\mathbf{h}(\Delta(\zeta, \mathbf{d}_0)))}\Big]^{\frac{1}{2}} \leq C_4(\mathbf{h}) 2^{(k+1)\frac{1+b(\mathbf{h})+\varepsilon}{1-b(\mathbf{h})-\varepsilon}}. \nonumber
\end{equation}
Since
$$ {\text {Area}}(\mathbf{h}(\Delta(\zeta, \mathbf{d}_0))) \leq C_5(\mathbf{h}) [{\textup{dist}}(\mathbf{h}(\zeta), \partial \mathbf{h}(\Delta(\zeta, \mathbf{d}_0)))]^2, $$
and
\begin{equation}
{\textup{dist}}(\mathbf{h}(\zeta), \partial {\mathbf{h}}(\Delta(\zeta, \mathbf{d}_0))) \leq 4 |{\mathbf{h}}'(\zeta)| {\textup{dist}}(\zeta, f(\mathbb{T}))=4 |{\mathbf{h}}'(\zeta)|\mathbf{d}_0, \nonumber
\end{equation}
by Remark \ref{re-1}. Consequently, we have
\begin{equation}
[{\text {Area}}(\mathbf{h}(\Delta(\zeta, 2^{k+1}\mathbf{d}_0)))]^{\frac{1}{2}} \leq  C_6(\mathbf{h}) 2^{(k+1) \frac{1+b(\mathbf{h})+\varepsilon}{1-b(\mathbf{h})-\varepsilon}}\cdot |{\mathbf{h}}'(\zeta)|\mathbf{d}_0.
\nonumber
\end{equation}
Since $\varepsilon<\frac{1}{3}(1-\|\mu_f\|_{\infty})-b(\mathbf{h})$ we have $b(\mathbf{h})+\varepsilon < \frac{1}{3}(1-\|\mu_f\|_{\infty})$. Then it follows from (\ref{Ik}) that
\begin{eqnarray}\label{xiugai-1}
\mathbf{I}_k & \leq & C_6(\mathbf{h}) \frac{b(\mathbf{h})+\varepsilon}{\sqrt{1-[b(\mathbf{h})+\varepsilon]^2}}\frac{2\sqrt{\pi}}{[2^k\mathbf{d}_0]^2} 2^{(k+1) \frac{1+b(\mathbf{h})+\varepsilon}{1-b(\mathbf{h})-\varepsilon}}\cdot |{\mathbf{h}}'(\zeta)|\mathbf{d}_0\nonumber \\
&\leq & 2\sqrt{\pi} C_6(\mathbf{h}) \frac{b(\mathbf{h})+\varepsilon}{\sqrt{1-\frac{1}{9}(1-\|\mu_f\|_{\infty})^2}}\cdot  2^{(k+1) \frac{1+b(\mathbf{h})+\varepsilon}{1-b(\mathbf{h})-\varepsilon}-2k}
\cdot \frac{|{\mathbf{h}}'(\zeta)|}{\mathbf{d}_0}\nonumber \\
&\leq & C_7(f, \mathbf{h}) \frac{b(\mathbf{h})+\varepsilon}{2^{k\frac{1-3[b(\mathbf{h})+\varepsilon]}{1-[b(\mathbf{h})+\varepsilon]}}}\cdot \frac{|{\mathbf{h}}'(\zeta)|}{\mathbf{d}_0}. \nonumber
\end{eqnarray}
Therefore, we have
\begin{eqnarray}
\lefteqn{\iint_{\Delta(\zeta, \mathbf{r})-f(\Delta)}\frac{|\bar{\partial} \mathbf{h}(w)|}{|w-\zeta|^3}dudv=\sum_{k=0}^{{\mathbf{m}}} \mathbf{I}_k} \nonumber  \\
 &&\leq C_7(f, \mathbf{h})[b(\mathbf{h})+\varepsilon] \frac{|\mathbf{h}'(\zeta)|}{\mathbf{d}_0}\cdot   \sum_{k=0}^{\infty} \frac{1}{2^{k\frac{1-3[b(\mathbf{h})+\varepsilon]}{1-[b(\mathbf{h})+\varepsilon]}}}\leq  C_{8}(f, \mathbf{h})[b(\mathbf{h})+\varepsilon] \frac{|{\mathbf{h}}'(\zeta)|}{\mathbf{d}_0}, \nonumber
\end{eqnarray}
since $\varepsilon< \frac{1}{3}(1-\|\mu_f\|_{\infty})-b(\mathbf{h})$ so that $$1-3[b(\mathbf{h})+\varepsilon]>1-3b(\mathbf{h})-[1-\|\mu_f\|_{\infty}-3b(\mathbf{h})]=\|\mu_f\|_{\infty}\geq 0.$$
Then, it follows from (\ref{key-1}) that
\begin{eqnarray}\label{key-2}|{\mathbf{h}}''(\zeta)| \leq  \frac{8\mathbf{M}}{{\mathbf{d}}^2}+C_{9}(f, \mathbf{h})[b(\mathbf{h})+\varepsilon]\frac{|{\mathbf{h}}'(\zeta)|}{\mathbf{d}_0}. \nonumber
\end{eqnarray}
Furthermore we obtain that
\begin{eqnarray}\label{int-l}
|N_{\mathbf{h}}(\zeta)|{\textup{dist}}(\zeta, f(\mathbb{T}))&=&\frac{|{\mathbf{h}}''(\zeta)|}{|\mathbf{h}'(\zeta)|}\mathbf{d}_0 \\
& \leq &  \frac{8\mathbf{M}}{{\mathbf{d}}^2}\frac{\mathbf{d}_0}{|\mathbf{h}'(\zeta)|}+C_{9}(f, \mathbf{h})[b(\mathbf{h})+\varepsilon]. \nonumber
\end{eqnarray}

To continue the proof, we need the following claim.
\begin{claim}
We have
\begin{equation}\label{cla-00}
\frac{\mathbf{d}_0}{|{\mathbf{h}}'(\zeta)|}=\frac{{\textup{dist}}(\zeta, f(\mathbb{T}))}{|{\mathbf{h}}'(\zeta)|} \rightarrow 0, \,\, {\text {as}}\,\,\,  {\textup{dist}}(\zeta, f(\mathbb{T})) \rightarrow 0.
\end{equation}
\end{claim}

\begin{proof}[Proof of the Claim]
Let $\zeta=f(z)$, $g(z)=\mathbf{h}\circ f(z)={\mathbf{h}}(\zeta)$. It is easy to see that $g(0)=g'(0)-1=0$. Then from Lemma \ref{l5} we have
\begin{equation}
|f'(z)|=|({\mathbf{h}}^{-1} \circ g(z))'|=|({\mathbf{h}}^{-1})'\circ g(z)|\cdot |g'(z)|=\frac{|g'(z)|}{|{\mathbf{h}}'(\zeta)|}\leq C_{10}(K_f)(1-|z|^2)^{-\|\mu_{f}\|_{\infty}}.  \nonumber
\end{equation}
Consequently, we have
\begin{eqnarray}\label{cla-0}
\frac{{\textup{dist}}(\zeta, f(\mathbb{T}))}{|{\mathbf{h}}'(\zeta)|} &\leq & 4 \frac{(1-|z|^2)|f'(z)|}{|{\mathbf{h}}'(\zeta)|} \\
& \leq &  4C_{10}(K_f) (1-|z|^2)^{1-\|\mu_f\|_{\infty}} \cdot \frac{|f'(z)|}{|g'(z)|}.\nonumber
\end{eqnarray}

Now let $\pi_1$ be a conformal mapping from $f(\Delta(R_2))$ into $\Delta$ with $\pi_1(0)=0$ and $\pi_2$ be a conformal mapping from $g(\Delta(R_2))$ into $\Delta$ with $\pi_2(0)=0$ . Since $f(\overline{\Delta})$ is contained in  $f(\Delta(R_2))$, we know that $\pi_1$ is a bi-Lipschitz mapping from $f(\overline{\Delta})$ to its image.  Then we know that
\begin{equation}\label{cla-1} |{\pi_1}(f(z_1))-{\pi_1}(f(z_2))| \asymp  |f(z_1)-f(z_2)|,\end{equation}
for any two different points $z_1, z_2$ in $\overline{\Delta}$. Similarly, we have
\begin{equation}\label{cla-2}  |{\pi_2}(g(z_1))-{\pi_2}(g(z_2))|\asymp |g(z_1)-g(z_2)|,\end{equation}
for any two different points $z_1, z_2$ in $\overline{\Delta}$. On the other hand, since $\pi_1 \circ {\mathbf{h}}^{-1} \circ {\pi_2}^{-1}$ is a quasiconformal mapping from $\Delta$ to itself fixing the origin.  We see from Lemma \ref{l4} that
\begin{eqnarray}\label{cla-3}
\lefteqn{|\pi_1 \circ {\mathbf{h}}^{-1} \circ {\pi_2}^{-1}\circ {\pi_2}\circ g(z_1)-\pi_1 \circ {\mathbf{h}}^{-1}  \circ {\pi_2}^{-1}\circ {\pi_2}\circ g(z_2)|}\\
&&\quad\quad= |\pi_1 \circ f(z_1)-\pi_1 \circ f(z_2)| \nonumber \\
&&\quad\quad \leq 16 |{\pi_2}\circ g(z_1)-{\pi_2}\circ g(z_2)|^{\frac{1}{{K_0}}},\,\, z_1, z_2 \in \overline{\Delta}. \nonumber
\end{eqnarray}
Here ${K}_0$ is the dilatation of the quasiconformal mapping $\pi_1 \circ {\mathbf{h}}^{-1} \circ {\pi_2}^{-1}$. It is easy to see that
\begin{equation}\label{K}
\frac{1-b(\mathbf{h})-\varepsilon}{1+b(\mathbf{h})+\varepsilon}\leq \frac{1}{{K_0}}\leq 1. \end{equation}
Then, we obtain from (\ref{cla-1})-(\ref{cla-3}) that
\begin{equation}\label{cla-eq}
 |f(z_1)-f(z_2)| \leq C_{11}(\pi_1, \pi_2) |g(z_1)-g(z_2)|^{\frac{1}{{K_0}}}, \,\, z_1, z_2\in \overline{\Delta}.
\end{equation}
Meanwhile, we know by Proposition \ref{pro-1} and Remark \ref{re-2} that
$$(1-|z|^2)|f'(z)| \asymp {\textup{dist}}(f(z), f(\mathbb{T}))\asymp |f(z)-f(e^{i\arg{z}})|, $$
$$(1-|z|^2)|g'(z)| \asymp {\textup{dist}}(g(z), g(\mathbb{T}))\asymp |g(z)-g(e^{i\arg{z}})|, $$
and by Remark \ref{remark} that
\begin{equation}\label{add} |g(z)-g(e^{i\arg z})| \leq C_{12}(g)(1-|z|^2)^{\frac{2}{1+K_{g}}}.\end{equation}
It follows from (\ref{cla-eq}) and (\ref{add}) that
\begin{equation}
\frac{|f'(z)|}{|g'(z)|} \asymp \frac{ |f(z)-f(e^{i\arg{z}})|}{|g(z)-g(e^{i\arg{z}})|}\leq C_{13}(f, \mathbf{h}) (1-|z|^2)^{-\frac{2}{1+K_g} \cdot (1-\frac{1}{{K_0}})},\,\, z\in \Delta. \nonumber
\end{equation}
Thus, from (\ref{cla-0}), we get that
\begin{eqnarray}\label{cla-l}
\frac{{\textup{dist}}(\zeta, f(\mathbb{T}))}{|\mathbf{h}'(\zeta)|}  \leq  C_{14}(f, \mathbf{h}) (1-|z|^2)^{1-\|\mu_f\|_{\infty}-\frac{2}{1+K_g}\cdot (1-\frac{1}{K_0})}.
\end{eqnarray}
From (\ref{K}), we note that
\begin{eqnarray}\lefteqn{1-\|\mu_f\|_{\infty}-\frac{2}{1+K_g}\cdot (1-\frac{1}{{K_0}}) \geq 1-\|\mu_f\|_{\infty}-(1-\frac{1}{{K_0}}) }\nonumber \\
&&= \frac{1}{{K_0}}-\|\mu_f\|_{\infty}  \geq  \frac{(1-\|\mu_f\|_{\infty})-(1+\|\mu_f\|_{\infty})[b(\mathbf{h})+\varepsilon]}{1+b(\mathbf{h})+\varepsilon}>0,\nonumber  \end{eqnarray}
since $\varepsilon< \frac{1}{3}(1-\|\mu_f\|_{\infty})-b(\mathbf{h})$ so that
$$(1+\|\mu_f\|_{\infty})[b(\mathbf{h})+\varepsilon]<\frac{2}{3}(1-\|\mu_f\|_{\infty})<1-\|\mu_f\|_{\infty}.$$
Then it follows from (\ref{cla-l}) that
\begin{equation}
\frac{{\textup{dist}}(\zeta, f(\mathbb{T}))}{|\mathbf{h}'(\zeta)|} \rightarrow 0, \,\, {\text {as}}\,\,\,  {\text {dist}}(\zeta, f(\mathbb{T})) \rightarrow 0, \nonumber
\end{equation}
since ${\textup{dist}}(\zeta, f(\mathbb{T})) \rightarrow 0$ is equivalent to $|z| \rightarrow 1^{-}$.  The claim is proved.
 \end{proof}

We proceed to prove Proposition \ref{key}.  From (\ref{int-l}) and (\ref{cla-00}) and their proof, we see that, for any $\varepsilon\in (0,  \frac{1}{3}(1-\|\mu_f\|_{\infty})-b(\mathbf{h}))$,  we can find a constant  $\delta>0$ such that
\begin{eqnarray}
|N_{\mathbf{h}}(\zeta)|{\textup{dist}}(\zeta, f(\mathbb{T})) \leq C(f, \mathbf{h})[b(\mathbf{h})+\varepsilon],\nonumber
\end{eqnarray}
for all $\zeta\in f(\Delta)$ with ${\textup{dist}}(\zeta, f(\mathbb{T}))<\delta$.  Now, we finish the proof of  Proposition \ref{key}.

\section{{\bf Integral means spectrum functionals on Teichm\"uller spaces}}
 In this section, we first recall the definitions of universal Teichm\"uller space and the  universal  asymptotic Teichm\"uller space and then restate some main theorems of this paper. For the references about the Teichm\"uller spaces, see \cite{GN, EGN-1, EMS, EGN-2, L, Zhur, AG1}.

We use $\mathcal{M}(\Delta^{*})$ to denote the open unit ball of the Banach space $\mathcal{L}^{\infty}({\Delta}^{*})$ of essentially bounded measurable functions in ${\Delta}^{*}$. For $\mu \in \mathcal{M}({\Delta}^{*})$, let $f_{\mu}$ be the quasiconformal mapping in the extended complex plane $\widehat{\mathbb{C}}$ with complex dilatation equal to $\mu$ in $\Delta^{*}$, equal to $0$ in $\Delta$, normalized $f_{\mu}(0)=0, \, f'_{\mu}(0)=1, \, f_{\mu}(\infty)=\infty$.  We say two elements $\mu$ and $\nu$ in $\mathcal{M}(\Delta^{*})$ are equivalent, denoted by $\mu\sim \nu$, if $f_{\mu}|_{\Delta}=f_{\nu}|_{\Delta}$. The equivalence class of $\mu$  is denoted by $[\mu]_{T}$. Then $T=\mathcal{M}(\Delta^{*})/\sim$ is one model of the universal Teichm\"uller space.

The Teichm\"uller distance $d_T([\mu]_{T}, [\nu]_{T})$ of two points $[\mu]_{T}$, $[\nu]_{T}$ in $T$ is defined as
\begin{eqnarray}
\lefteqn{d_T([\mu]_{T}, [\nu]_{T})=\frac{1}{2}\inf  \bigg\{ \log \frac{1+\|(\mu_1-\nu_1)/(1-\overline{\nu_1}\mu_1)\|_{\infty}}{1-\|(\mu_1-\nu_1)/(1-\overline{\nu_1}\mu_1)\|_{\infty}},}
\nonumber \\
&&\quad\quad\qquad\qquad\qquad\qquad\qquad\qquad [\mu_1]_{T}=[\mu]_{T}, [\nu_1]_{T}=[\nu]_{T}  \bigg\}. \nonumber
\end{eqnarray}

We say $\mu$ and $\nu$ in $\mathcal{M}({\Delta}^{*})$ are asymptotically equivalent if there exists some $\widetilde{\nu}$ such that $\widetilde{\nu}$ and $\nu$ are equivalent and $\widetilde{\nu}(z)-\mu(z) \rightarrow 0$ as $|z|\rightarrow 1^{+}$. The asymptotic equivalence
of $\mu$ will be denoted by $[\mu]_{AT}$. The {\em universal asymptotic Teichm\"uller space} $AT$ is the set of
all the asymptotic equivalence classes $[\mu]_{AT}$ of elements $\mu$ in $\mathcal{M}({\Delta}^{*})$. The Teichm\"uller distance $d_{AT}([\mu]_{AT}, [\nu]_{AT})$ of two points $[\mu]_{AT}$, $[\nu]_{AT}$ in $AT$ is defined as
\begin{eqnarray}
\lefteqn{d_{AT}([\mu]_{AT}, [\nu]_{AT})=\frac{1}{2}\inf  \bigg\{ \log \frac{1+H[(\mu_1-\nu_1)/(1-\overline{\nu_1}\mu_1)]}{1-H[(\mu_1-\nu_1)/(1-\overline{\nu_1}\mu_1)]},}
\nonumber \\
&&\quad\quad\qquad\qquad\qquad\qquad\qquad\qquad [\mu_1]_{AT}=[\mu]_{AT}, [\nu_1]_{AT}=[\nu]_{AT)}  \bigg\}. \nonumber
\end{eqnarray}
Here,
\begin{equation}\label{bn}
H[\mu]=\inf\{\|\mu|_{\Delta^{*}-E}\|_{\infty}: \, E {\text { is a compact set in}}\,\,\Delta^{*}\}.\end{equation}
\begin{remark}We can check from (\ref{qua}) that $B_b(t)=\sup\limits_{[\mu]_T\in T} \beta_{f_{\mu}}(t)$ for each $t\in \mathbb{R}$.
\end{remark}
We set
$$\Lambda_1: [\mu]_T  \mapsto  N_{f_{\mu}}, \,\, \Lambda_2: [\mu]_T  \mapsto  S_{f_{\mu}}.$$
The mapping $\Lambda_2$ is known as the {\em Bers mapping}. We call $\Lambda_1$ the {\em Pre-Bers mapping}.  It is well known that

\begin{proposition}\label{bers-2}
The mapping $\Lambda_2: [\mu]_T \mapsto  S_{f_\mu}$ from $(T, d_T)$  to its image $T_2$ in $E_2$ is a homeomorphism.
\end{proposition}

We also have  
\begin{proposition}\label{bers-1}
The mapping $\Lambda_1: [\mu]_T \mapsto  N_{f_{\mu}}$ from $(T, d_{T})$ to its image $T_1$ in $E_1$  is a homeomorphism.
\end{proposition}

\begin{remark}
This proposition seems to be known in the literature. For the convenience of readers, we will present a detailed proof for this result. 
\end{remark}

\begin{proof}[Proof of Proposition \ref{bers-1}]
It is obvious that $\Lambda_1$ is bijective. We will first show that $\Lambda_1$ is continuous. For $\mu, \nu \in \mathcal{M}(\Delta^*)$. We may assume that $\|\mu-\nu\|_{\infty}>0$, and let $\mathbf{f}=f_{\nu}\circ f_{\mu}^{-1}$.  Then we see that $\mathbf{f}(0)=\mathbf{f}'(0)-1=0, \mathbf{f}(\infty)=\infty$ and from (\ref{dila}) that 
\begin{equation}
\mu_{\mathbf{f}}\circ f_{\mu}(z) =\frac{1}{\chi}\frac{\nu(z)-\mu(z)}{1-\overline{\mu(z)}\nu(z)},\,\, \chi=\frac{\overline{\partial f_{\mu}}}{\partial f_{\mu}}, \,\,\, z\in \Delta.\nonumber
\end{equation}
Note that $\|\mu_{\mathbf{f}}\|_{\infty}>0$. From \cite{AB}, we know that there is a unique quasiconformal mapping $f(z, t)$ from $\widehat{\mathbb{C}}$ to itself such that $f(0, t)=f_{z}(0, t)-1=0,f(\infty, t)=\infty$ and 
$f_{\bar{z}}(z, t)=t\mu_{\mathbf{f}}f_z(z, t)$
for each $t \in \mathcal{D}_{\mathbf{f}}=\{t: |t|<1/\|\mu_{\mathbf{f}}\|_{\infty}\}.$ 

For each $t\in \mathcal{D}_{\mathbf{f}}$, we see that $f(z, t)$ is conformal in $f_{\mu}(\Delta)$.  For fixed $z\in f_{\mu}(\Delta)$, the function 
$\Phi(t)=N_{f(z, t)}(z) \rho_{f_{\mu}(\Delta)}^{-1}(z)$ is holomorphic in $\mathcal{D}_{\mathbf{f}}$. 

We see from \cite{O} that $|\Phi(t)|\leq 8$. Note that $f(z, 0)=z$ so that $\Phi(0)=0$, then by using Schwarz's lemma, we obtain that $|\Phi(t)|\leq 8t\|\mu_{\mathbf{f}}\|_{\infty}$ for $t\in \mathcal{D}_{\mathbf{f}}$. On the other hand, from $f(z, 1)=\mathbf{f}(z)$, we have
$|N_{\mathbf{f}}(z)|\rho_{f_{\mu}(\Delta)}^{-1}(z) \leq 8 \|\mu_{\mathbf{f}}\|_{\infty}$ for all $z\in f_{\mu}(\Delta)$. 
It follows from the fact $\mathbf{f}=f_{\nu}\circ f_{\mu}^{-1}$ and (\ref{chain}) that 
$$\|N_{f_{\mu}}(z)-N_{f_{\nu}}(z)\|_{E_1}=\sup\limits_{z\in f_{\mu}(\Delta)}|N_{\mathbf{f}}(z)|\rho_{f_{\mu}(\Delta)}^{-1}(z) \leq 8 \|(\nu(z)-\mu(z))/(1-\overline{\mu(z)}\nu(z))\|_{\infty}.$$
Consequently, we obtain that
\begin{eqnarray}\|N_{f_{\mu}}(z)-N_{f_{\nu}}(z)\|_{E_1} &\leq & 8 \inf\limits_{\mu_1 \sim \mu, \nu_1 \sim \mu} \|(\nu_1(z)-\mu_1(z))/(1-\overline{\mu_1(z)}\nu_1(z))\|_{\infty}\nonumber 
\\ &\leq &  8 \inf\limits_{\mu_1 \sim \mu, \nu_1 \sim \mu} \log \frac{1+\|(\nu_1-\mu_1)/(1-\overline{\mu_1}\nu_1)\|_{\infty}}{1-\|(\nu_1-\mu_1)/(1-\overline{\mu_1}\nu_1)\|_{\infty}} \nonumber
\\ &=& 16d_{T}([\mu]_T, [\nu]_T). \nonumber
\end{eqnarray}
This implies that $\Lambda_1$ is continuous. 

We next prove that $\Lambda_1^{-1}$ is continuous.  Consider the mapping $\Gamma: \phi \mapsto \phi'-\frac{1}{2}\phi^2$. Note that $\Gamma(N_{f_{\mu}})=S_{f_{\mu}}$ for any $\mu\in \mathcal{M}(\Delta^*)$,  we know from \cite{Zhur} that $\Gamma$ continuously maps $T_1$ to $T_2$. In view of  the fact that $\Lambda_1^{-1}(N_{f_{\mu}})=\Lambda_2^{-1} \circ \Gamma(N_{f_{\mu}})$, we conclude from Proposition \ref{bers-2} that 
$\Lambda_1^{-1}$ is continuous. Proposition \ref{bers-1} is proved. 
\end{proof}

\begin{remark}
In view of Proposition \ref{bers-2} and \ref{bers-1}, we can identify the universal Teichm\"uller space with $T_1$ or $T_2$.
Let $\mathcal{S}_q^{\infty}$ be the class of all functions $f\in \mathcal{S}_q$ that have a quasiconformal extension (still denoted by $f$) to $\widehat{\mathbb{C}}$ with $f(\infty)=\infty$. We set
$$\mathbf{N}_q: =\{\phi: \phi=N_f(z), f\in \mathcal{S}_q^{\infty}\}, \,\,\, \mathbf{S}_q:=\{\phi: \phi=S_f(z), f\in \mathcal{S}_q^{\infty}\}.$$
It is easy to see that $T_1=\mathbf{N}_q$ and $T_2= \mathbf{S}_q$.
\end{remark}

We will study the following IMS functional defined on $T$ and prove that
\begin{theorem}[{\bf Theorem \ref{m-1}}]\label{th-1}
For each $t\in \mathbb{R}$, the IMS functional $I_{{T}}: [\mu]_T  \mapsto \beta_{f_{\mu}}(t)$ on $T$ is continuous.
\end{theorem}

We next identify the universal Teichm\"uller space with $T_1=\mathbf{N}_q$. From Proposition \ref{pro-1}, since $\mathbf{N}$ is closed in $E_1$, then the closure $\overline{\mathbf{N}}_q$ of $\mathbf{N}_q$ is contained in $\mathbf{N}$.  We regard  $\overline{T}_1=\overline{\mathbf{N}}_q$ as a model of the closure of the universal Teichm\"uller space.
For any $\phi \in \overline{\mathbf{N}}_q$, there is a unique univalent function $f_{\phi}(z)$ with $f_{\phi}\in \mathcal{S}$ and such that
$$\phi(z)=N_{f_{\phi}}(z).$$ Actually,  we can take 
\begin{equation}\label{exa}f_{\phi}(z)=\int_{0}^{z} e^{\int_{0}^{\zeta}\phi(w) dw}d\zeta, \,\,\,z\in\Delta.\end{equation}

Moreover, we shall prove that
\begin{theorem}[{\bf Theorem \ref{m-3}}]\label{th-2}
For each $t\in \mathbb{R}$, the IMS functional $I_{\overline{T}_1}: \phi  \mapsto  \beta_{f_{\phi}}(t)$ on $\overline{T}_1$ is continuous.
\end{theorem}

For the IMS functional on the universal asymptotic Teichm\"uller space,  we shall show that
\begin{theorem}[{\bf Theorem \ref{m-2}}]\label{th-3}
For each $t\in \mathbb{R}$, the IMS functional $I_{{AT}}: [\mu]_{AT} \mapsto  \beta_{f_\mu}(t)$ is well-defined and continuous on $AT$.
\end{theorem}

\section{\bf {Proof of Theorem \ref{th-1}, \ref{th-2} and \ref{th-3}} }
To prove these theorems, we shall recall some known lemmas and establish some new ones. We will use the following criterion for the integral means spectrum.  For $\alpha>-1$, we define the Hilbert space $\mathcal{H}_{\alpha}^2(\Delta)$ as
$$\mathcal{H}_{\alpha}^2(\Delta)=\{\phi \in \mathcal{A}(\Delta) : \|\phi\|_{\alpha}^2:=(\alpha+1)\iint_{\Delta}|\phi (z)|^2(1-|z|^2)^{\alpha}\frac{dxdy}{\pi}<\infty\}.$$
It is known,  see \cite{HS-1},  that

\begin{lemma}\label{cri} Let $\alpha>-1$.  For each $t\in \mathbb{R}$, we have
$$\beta_{f}(t)=\inf\{ \alpha+1: \,\,(f')^{\frac{t}{2}}\in \mathcal{H}_{\alpha}^2(\Delta)\}.$$
\end{lemma}

We also need the following results.
\begin{lemma}\label{key-est}
Let $f, g \in \mathcal{S}$.  For $\varepsilon>0$, there is a constant $r\in (0, 1) $ such that
\begin{equation}\label{condition}\sup\limits_{|z|\in (r, 1)}|N_g(z)-N_f(z)|(1-|z|^2) <\varepsilon. \end{equation}
Then there exist two positive numbers $C_1(r, \varepsilon)$ and $C_2(r,\varepsilon)$ such that
\begin{equation}
C_1(r, \varepsilon)\Big( \frac{1-|z|}{1+|z|}\Big)^{\frac{\varepsilon}{2}}\leq |{\mathbf{h}}' \circ f(z)| \leq  C_2(r, \varepsilon) \Big( \frac{1+|z|}{1-|z|}\Big)^{\frac{\varepsilon}{2}},\nonumber
\end{equation}
for all $|z|\in (r,1)$.  Here $\mathbf{h}=g \circ f^{-1}$.\end{lemma}

\begin{proof}
From $\mathbf{h}=g \circ f^{-1}$ and (\ref{chain}), we have
 \begin{equation}
N_g(z)-N_f(z)=\frac{{\mathbf{h}}'' \circ f(z) }{{\mathbf{h}}' \circ f(z)} \cdot f'(z).\nonumber
\end{equation}
We let $$P(z)=\frac{{\mathbf{h}}'' \circ f(z)}{{\mathbf{h}}' \circ f(z)} \cdot f'(z),\,\, L(z)={\mathbf{h}}' \circ f(z).$$
Let $z=|z|e^{i \arg z}$ be such that $|z|\in (r,1)$, then
\begin{equation}
\log L(z)=\int_{z_r}^z P(\zeta)d\zeta+ \log L(z_r), \nonumber
\end{equation}
where $z_r=r e^{i \arg z}$, and the integral is taken on the radial path from $z_r$ to $z$. 

On the other hand, since for all $|z| \in (r, 1)$,  $|P(z)(1-|z|^2)| <\varepsilon$, then we have
\begin{eqnarray}
\Big| \int_{z_r}^z P(\zeta)d\zeta\Big| &=& \Big| \int_{r}^{|z|} P(t e^{i \arg z}) e^{i \arg z}dt \Big| \nonumber \\
 &= & \Big| \int_{r}^{|z|} P(t e^{i \arg z})(1-t^2) \cdot\frac{1}{1-t^2} e^{i \arg z}dt \Big| \nonumber  \\
 &\leq & \int_{r}^{|z|} \frac{\varepsilon}{1-t^2}dt =\frac{\varepsilon}{2}\Big[\log \frac{1+|z|}{1-|z|}-\log \frac{1+r}{1-r}\Big]. \nonumber \end{eqnarray}
We denote
$$\mathbf{M}_0=\max\limits_{|z|=r}| \log  L(z)|=\max\limits_{|z|=r} |\log{\mathbf{h}}' \circ f(z)|.$$Then we  see from the fact $\Big|\log |L(z)|\Big|\leq |\log L(z)|$ that
\begin{equation}
\Big|\log |{\mathbf{h}}' \circ f(z)|\Big| \leq \frac{\varepsilon}{2}\Big[\log \frac{1+|z|}{1-|z|}-\log \frac{1+r}{1-r}\Big]+\mathbf{M}_0. \nonumber
\end{equation}
It follows that
\begin{equation}
e^{-\mathbf{M}_0}\Big( \frac{1+r}{1-r}\Big)^{\frac{\varepsilon}{2}} \Big( \frac{1-|z|}{1+|z|}\Big)^{\frac{\varepsilon}{2}}\leq |{\mathbf{h}}' \circ f(z)| \leq  e^{\mathbf{M}_0}\Big( \frac{1-r}{1+r}\Big)^{\frac{\varepsilon}{2}} \Big( \frac{1+|z|}{1-|z|}\Big)^{\frac{\varepsilon}{2}}. \nonumber
\end{equation}
This proves the lemma.
\end{proof}

\begin{lemma}\label{key-lemma}
Let $f, g \in \mathcal{S}$ and $t\neq 0$. (1) If $\beta_f(t):=\gamma>0$ and for $\varepsilon\in (0, \gamma/|t|)$ there is a constant $r\in (0, 1) $ such that (\ref{condition}) holds, then we have
\begin{equation}
|\beta_g(t)-\beta_f(t)|\leq |t|\varepsilon,\,\,\, {\textup i.e.,}\,\,\, \gamma-|t|\varepsilon \leq \beta_g(t)\leq \gamma+|t|\varepsilon
\end{equation}
(2) If $\beta_f(t)=0$ and for $\varepsilon>0$ there is a constant $r\in (0,1)$ such that (\ref{condition}) holds, then we have $\beta_g(t)\leq |t|\varepsilon$.
\end{lemma}

\begin{proof}
(1) Let $\mathbf{h}=g\circ f^{-1}$. When $\beta_{f}(t)=\gamma >0$, by Lemma \ref{key-est}, for $\varepsilon\in (0, \gamma/|t|)$, we have
\begin{equation}\label{we-1}
C_1(r, \varepsilon)\Big( \frac{1-|z|}{1+|z|}\Big)^{\frac{\varepsilon}{2}}\leq |{\mathbf{h}}' \circ f(z)| \leq  C_2(r, \varepsilon) \Big( \frac{1+|z|}{1-|z|}\Big)^{\frac{\varepsilon}{2}}, \, |z|\in (r,1).
\end{equation}

On the other hand, in view of Lemma \ref{cri}, we see that, for $\varepsilon\in (0, \gamma/|t|)$,
\begin{equation}\label{we-2}
\iint_{\Delta}|f'(z)|^t (1-|z|^{2})^{-1+\gamma+|t|\varepsilon/2}dxdy<\infty,
\end{equation}
and
\begin{equation}\label{we-3}
\iint_{\Delta}|f'(z)|^t (1-|z|^{2})^{-1+\gamma-|t|\varepsilon/2}dxdy=\infty.
\end{equation}
When $t>0$, it follows from the second inequality of (\ref{we-1}) and (\ref{we-2}) that
\begin{eqnarray}
\lefteqn{\iint_{\Delta-\Delta(r)}|g'(z)|^t (1-|z|^{2})^{-1+\gamma+t\varepsilon}dxdy} \nonumber \\&&=\iint_{\Delta-\Delta(r)}|\mathbf{h}'\circ f(z)|^t |f'(z)|^t (1-|z|^{2})^{-1+\gamma+t\varepsilon}dxdy  \nonumber \\
&&\leq [C_2(r, \varepsilon)]^t \iint_{\Delta-\Delta(r)}\Big( \frac{1+|z|}{1-|z|}\Big)^{t\varepsilon/2} |f'(z)|^t (1-|z|^{2})^{-1+\gamma+t\varepsilon}dxdy \nonumber \\
&&\leq 2^{t\varepsilon}[C_2(r, \varepsilon)]^t\iint_{\Delta-\Delta(r)} |f'(z)|^t (1-|z|^{2})^{-1+\gamma+t\varepsilon/2}dxdy<\infty. \nonumber
\end{eqnarray}
Then it is easy to see from Lemma \ref{cri} that $\beta_{g}(t)\leq \gamma+t\varepsilon.$ Meanwhile, from the first  inequality of (\ref{we-1}) and (\ref{we-3}), we have
\begin{eqnarray}
\lefteqn{\iint_{\Delta-\Delta(r)}|g'(z)|^t (1-|z|^{2})^{-1+\gamma-t\varepsilon}dxdy} \nonumber \\&&=\iint_{\Delta-\Delta(r)}|\mathbf{h}'\circ f(z)|^t |f'(z)|^t (1-|z|^{2})^{-1+\gamma-t\varepsilon}dxdy  \nonumber \\
&&\geq [C_1(r, \varepsilon)]^t \iint_{\Delta-\Delta(r)}\Big( \frac{1-|z|}{1+|z|}\Big)^{t\varepsilon/2} |f'(z)|^t (1-|z|^{2})^{-1+\gamma-t\varepsilon}dxdy \nonumber \\
&&\geq 2^{-t\varepsilon}[C_1(r, \varepsilon)]^t\iint_{\Delta-\Delta(r)} |f'(z)|^t (1-|z|^{2})^{-1+\gamma-t\varepsilon/2}dxdy=\infty. \nonumber
\end{eqnarray}
This implies that $\beta_{g}(t)\geq \gamma-t\varepsilon$. Hence we have $|\beta_{g}(t)-\gamma|\leq t\varepsilon$ when $t>0$. The case $t<0$ can be proved by the similar way.

(2) When $\beta_f(t)=0$, for any $\varepsilon>0$, repeating the above arguments by only using the second inequality of (\ref{we-1}) and (\ref{we-2}), we can prove that $\beta_g(t)\leq |t|\varepsilon.$ Now, the proof of Lemma \ref{key-lemma} is finished. \end{proof}

\subsection{Proof of Theorem \ref{th-1} and \ref{th-2}}
We see from Proposition \ref{bers-1} that Theorem \ref{th-2} implies Theorem \ref{th-1}.  We will only prove Theorem \ref{th-2}.
\begin{proof}[Proof of Theorem \ref{th-2}]
Since the case for $t=0$ is trivial, we assume that $t\neq0$.  For any $\phi\in \overline{\mathbf{N}}_q$, we take $f_{\phi}$ as in (\ref{exa}). For given $\psi\in  \overline{\mathbf{N}}_q$. When $\beta_{f_{\psi}}(t)=\gamma>0$,  to prove $I_{\overline{T}_1}$ is continuous at $\psi$,  it suffices to prove that, for small $\varepsilon>0$, there is a constant $\delta>0$ such that $$|\beta_{f_{\phi}}(t)-\gamma|\leq \varepsilon,$$
for any $\phi \in\overline{\mathbf{N}}_q$ satisfying that  $\|\phi-\psi\|_{E_1}<\delta$.

Actually,  for any $\varepsilon\in (0, \gamma)$, we take $\delta=\varepsilon/|t|$. Let $\phi \in\overline{\mathbf{N}}_q$ satisfy that $\|\phi-\psi\|_{E_1}<\delta=\varepsilon/|t|$. Then, for  any number $r\in (0,1)$, we have
\begin{equation}
\sup\limits_{|z|\in (r,1)} |N_{f_{\phi}}(z)-N_{{f}_{\psi}}(z)|(1-|z|^2)<\varepsilon/|t|.\nonumber
\end{equation}
Hence, by (1) of Lemma \ref{key-lemma}, we have
$$|\beta_{f_{\phi}}(t)-\gamma|\leq |t|\cdot\varepsilon/|t|=\varepsilon.$$
This proves that $I_{\overline{T}_1}$ is continuous at $\psi$ when $\beta_{f_{\psi}}(t)>0$.

When $\beta_{f_{\psi}}(t)=0$, for any $\varepsilon>0$, we still take $\delta=\varepsilon/|t|$. Similarly, by using (2) of Lemma \ref{key-lemma}, we have $\beta_{f_{\phi}}(t)\leq \varepsilon$ for any $\phi \in\overline{\mathbf{N}}_q$ satisfying that  $\|\phi-\psi\|_{E_1}<\delta$.  This means that $I_{\overline{T}_1}$ is continuous at $\psi$ when $\beta_{f_{\psi}}(t)=0$. The proof of Theorem \ref{th-2} is complete.   \end{proof}

\subsection{Proof of Theorem \ref{th-3}} To prove Theorem \ref{th-3}, we shall establish some new results about the universal asymptotic Teichm\"uller space. We define the closed subspace $E_{1, 0}$ of $E_1$ as
$$E_{1,0}:=\{\phi\in E_1: \lim\limits_{|z|\rightarrow 1^{-}}\phi(z)(1-|z|^2)=0\}.$$
The closed subspace $E_{2,0}$ of $E_2$ is defined as
$$E_{2,0}:=\{\phi\in E_2: \lim\limits_{|z|\rightarrow 1^{-}}\phi(z)(1-|z|^2)^2=0\}.$$

Two elements $\phi_1, \phi_2 \in E_j$ are said to be equivalent, if $\phi_1-\phi_2\in E_{j,0}, j=1,2$. The equivalence class of $\phi\in E_j$  is denoted by $[\phi]_{E_j}, j=1,2$.  The set of all equivalence classes $[\phi]_{E_j}$ will be denote by $E_j/E_{j,0}, j=1,2$, respectively. $E_j/E_{j,0}$ is a Banach space with the quotient norm
$$\|[\phi]_{E_j}\|:=\inf\limits_{\psi \in [\phi]_{E_j}} \|\psi\|_{E_j}=\inf\limits_{\psi \in E_{j,0}} \|\phi+\psi\|_{E_j},\, j=1,2.$$

The following description of the asymptotically equivalence in terms of Schwarzian derivative has been given in \cite{EMS}.
\begin{proposition}\label{ase-1}
Let $\mu, \nu \in \mathcal{M}({\Delta}^{*})$. $\mu$ is  asymptotically equivalent to $\nu$ if and only if $S_{f_{\nu}}-S_{f_{\mu}}$ belongs to $E_{2,0}$.
\end{proposition}
We will give a new characterization of the asymptotically equivalence in terms of Pre-Schwarzian derivative.
We shall prove that
\begin{proposition}\label{ase-2}
Let $\mu, \nu \in \mathcal{M}({\Delta}^{*})$. $\mu$ is  asymptotically equivalent to $\nu$ if and only if $N_{f_{\nu}}-N_{f_{\mu}}$ belongs to $E_{1,0}$.
\end{proposition}

\begin{proof}
We first prove the if part.  Let $\mathbf{h}=f_{\nu}\circ f_{\mu}^{-1}$.  When $N_{f_{\nu}}-N_{f_{\mu}}\in E_{1,0}$, we see from
\begin{equation}\label{eq-1}
N_{f_{\nu}}-N_{f_{\mu}}=(N_{\mathbf{h}}\circ f_{\mu})\cdot f'_{\mu}
\end{equation}
that $\log {\mathbf{h}}'\circ f_{\mu}(z)$ belongs to the little Bloch space $\mathcal{B}_0$, which is defined as
$${\mathcal{B}_0}:=\{\phi\in \mathcal{A}(\Delta):\, \lim\limits_{|z|\rightarrow 1^{-}}\phi'(z)(1-|z|^2)=0\}.$$
From \cite[Proposition 8]{Zh-1}, we see that
\begin{equation}
[\log {\mathbf{h}}'\circ f_{\mu}(z)]''(1-|z|^2)^2 \rightarrow 0,\,\, {\text as }\,\, |z| \rightarrow {1^{-}}. \nonumber
\end{equation}
That is
\begin{eqnarray}
\lefteqn{\Big[\frac{{\mathbf{h}}''(\zeta)}{{\mathbf{h}}'(\zeta)}\Big]'\circ f_{\mu}(z)\cdot [f'_{\mu}(z)]^2(1-|z|^2)^2}\nonumber \\
&&\quad+\frac{{\mathbf{h}}''\circ f_{\mu}(z)}{{\mathbf{h}}'\circ f_{\mu}(z)}\cdot [f''_{\mu}(z)](1-|z|^2)^2 \rightarrow 0,\,\, {\text as }\,\, |z| \rightarrow {1^{-}}. \nonumber
\end{eqnarray}
Note that
\begin{eqnarray}
\lefteqn{\frac{{\mathbf{h}}''\circ f_{\mu}(z)}{{\mathbf{h}}'\circ f_{\mu}(z)}\cdot [f''_{\mu}(z)](1-|z|^2)^2}\nonumber \\
&&=[N_{\mathbf{h}}\circ f_{\mu}(z)][f'_{\mu}(z)](1-|z|^2)\cdot \|N_{f_{\mu}}\|_{E_1}   \rightarrow 0,\,\, {\text as }\,\, |z| \rightarrow {1^{-}}. \nonumber
\end{eqnarray}
It follows that
\begin{eqnarray}\label{eq-2}
\Big[\frac{{\mathbf{h}}''(\zeta)}{{\mathbf{h}}'(\zeta)}\Big]'\circ f_{\mu}(z)\cdot [f'_{\mu}(z)]^2(1-|z|^2)^2  \rightarrow 0,\,\, {\text as }\,\, |z| \rightarrow {1^{-}}.
\end{eqnarray}
On the other hand, from (\ref{chain-1}), we have
\begin{equation}
S_{f_{\nu}}-S_{f_{\mu}}=(S_{\mathbf{h}}\circ f_{\mu})\cdot[f'_{\mu}]^2, \nonumber
\end{equation}
that is
\begin{eqnarray}
S_{f_{\nu}}(z)-S_{f_{\mu}}(z)=\Big[\frac{{\mathbf{h}}''(\zeta)}{{\mathbf{h}}'(\zeta)}\Big]'\circ f_{\mu}(z)\cdot[f'_{\mu}(z)]^2-\frac{1}{2}[N_{\mathbf{h}}\circ f_{\mu}(z)]^2\cdot[f'_{\mu}(z)]^2. \nonumber
\end{eqnarray}
Consequently, we see from (\ref{eq-1}) and (\ref{eq-2}) that $S_{f_{\nu}}-S_{f_{\mu}}$ belongs to $E_{2,0}$.  By Proposition \ref{ase-1}, we obtain that $\mu$ is  asymptotically equivalent to $\nu$. The if part is proved.

We continue to prove the only if part. When $\mu$ is  asymptotically equivalent to $\nu$,  we know that there is a $\widetilde{\nu}\in \mathcal{M}(\Delta^{*})$ such that $\widetilde{\nu}\sim \nu$ and $\widetilde{\nu}(z)-\mu(z) \rightarrow 0$ as $z \rightarrow 1^{-}$.
Let $\widetilde{\mathbf{h}}=f_{\widetilde{\nu}}\circ f_{\mu}^{-1}$.  It follows from (\ref{dila}) that
\begin{eqnarray}
|\mu_{\widetilde{\mathbf{h}}}\circ f_{{\mu}}(z)|=\frac{|\widetilde{\nu}(z)-{\mu}(z)|}{|1-\overline{\mu(z)}\widetilde{\nu}(z)|} \leq \frac{|\widetilde{\nu}(z)-{\mu}(z)|}{1-\|\mu\|_{\infty}\|\widetilde{\nu}\|_{\infty}}. \nonumber
\end{eqnarray}
Hence we have $b(\widetilde{\mathbf{h}})=0$.  Then we see from Proposition \ref{key} that
\begin{eqnarray}
|N_{f_{\nu}}(z)-N_{f_{\mu}}(z)|(1-|z|^2)&=&|N_{\mathbf{h}}\circ f_{\mu}(z)|\cdot |f'_{\mu}(z)|(1-|z|^2) \nonumber \\
& \leq &  4 |N_{\mathbf{h}}\circ f_{\mu}(z)|{\textup{dist}}(f_{\mu}(z), f(\mathbb{T})) \rightarrow 0, \,\, {\text as}\,\, |z| \rightarrow 1^{-}.\nonumber
\end{eqnarray}
This means that $N_{f_{\nu}}-N_{f_{\mu}}$ belongs to $E_{1,0}$. The only if part is proved. This finishes the proof of Proposition \ref{ase-2}.
\end{proof}

In the standard theory of universal asymptotic Teichm\"uller space, $AT$ is embedded into an open subset of a complex Banach space by using the Bers mapping induced by the Schwarzian derivative. We shall consider the mapping induced by the Pre-Schwarzian derivative.
We let
$$\widetilde{\Lambda}_1: [\mu]_{AT}  \mapsto  [N_{f_{\mu}}]_{E_1}, \,\,\,  \widetilde{\Lambda}_2: [\mu]_{AT}  \mapsto  [S_{f_{\mu}}]_{E_2}.$$
The mapping $\widetilde{\Lambda}_2$ is called {\em asymptotic Bers mapping}.  It was proved in \cite{EMS} that
\begin{proposition}\label{abers-1}
The mapping $\widetilde{\Lambda}_2: [\mu]_{AT}  \mapsto  [S_{f_{\mu}}]_{E_2}$ from $(AT, d_{AT})$ to $\widetilde{\mathbf{S}}_q$ in $E_2/E_{2,0}$ is a homeomorphism.
Here, $$\widetilde{\mathbf{S}}_q:=\{[\phi]_{E_2}: \phi=S_f(z),\, f\in \mathcal{S}_q^{\infty}\}$$
is an open subset of $E_2/E_{2,0}$.
\end{proposition}

We will prove that
\begin{proposition}\label{abers-2}
The mapping $\widetilde{\Lambda}_1: [\mu]_{AT}  \mapsto  [N_{f_{\mu}}]_{E_1}$ from  $(AT, d_{AT})$ to $\widetilde{\mathbf{N}}_q$ in $E_1/E_{1,0}$ is a homeomorphism.
Here, $$\widetilde{\mathbf{N}}_q:=\{[\phi]_{E_1}: \phi=N_f(z),\, f\in \mathcal{S}_q^{\infty}\}$$
is an open subset of $E_1/E_{1,0}$.
\end{proposition}
\begin{remark}
The mapping $\widetilde{\Lambda}_1$ is called {\em asymptotic Pre-Bers mapping}.
\end{remark}

For $\mu \in \mathcal{M}(\Delta^{*})$, we know from Proposition \ref{ase-1} and \ref{ase-2} that the mapping
 $$\Xi: [N_{f_{\mu}}]_{E_1}  \mapsto   [S_{f_{\mu}}]_{E_2}$$
is well-defined and bijective from ${\widetilde{\mathbf{N}}_q}$ to
${\widetilde{\mathbf{S}}_q}$. Moreover, we have

\begin{lemma}\label{dengjia}
The mapping $\Xi: [N_{f_{\mu}}]_{E_1} \mapsto [S_{f_{\mu}}]_{E_2}$ is a homeomorphism from ${\widetilde{\mathbf{N}}_q}$ in $E_1/E_{1, 0}$ to
${\widetilde{\mathbf{S}}_q}$ in $E_2/E_{2, 0}$.
\end{lemma}

\begin{proof}
We let
 $$\mathcal{T}:=\{\phi: \phi=N_{f}(z), f\in \mathcal{S}_q\}.$$
$\mathcal{T}$ can be seen as one model of {\em universal Teichm\"uller curve}, see \cite{Bers, T}.  It is known from \cite{Zhur} that $\mathcal{T}$ is an open subset of $E_1$.  Let $P_{j}$ be the projection from $E_j$ to $E_j/E_{j,0},  j=1, 2$.  It is easy to see that $P_1(\mathcal{T})=P_1(\mathbf{N}_q)=\widetilde{\mathbf{N}}_q$.  Since $P_{1}$  is an open mapping, then we see that ${\widetilde{\mathbf{N}}_q}$ is an open subset of $E_1/E_{1, 0}$. 

Consider the mapping $\Gamma(\phi)=\phi'-\frac{1}{2}\phi^2$, we know that
$\Gamma$ is continuous from $\mathcal{T}$ to $\mathbf{S}_q$ and $\Gamma(\mathcal{T})=\Gamma(\mathbf{N}_q)=\mathbf{S}_q$.  We have the following commutative diagram.
\begin{displaymath}
    \xymatrix{
        \mathcal{T} \ar[r]^{\Gamma} \ar[d]^{P_1} & \mathbf{S}_q   \ar[d]^{P_2} \\
        \widetilde{\mathbf{N}}_q \ar[r]^{\Xi} & \widetilde{\mathbf{S}}_q  \ar[l]
    }
\end{displaymath}
Now, for any open subset $\mathbf{O}_S$ of $\widetilde{\mathbf{S}}_q$, we obtain that $P_2^{-1}(\mathbf{O}_S):=\widehat{\mathbf{O}}_S$ is open in  $\mathbf{S}_q$ since $P_2$ is continuous. Then we see that $\Gamma^{-1} \circ P_2^{-1}(\mathbf{O}_S)$ is open in $\mathcal{T}$. 
On the other hand, we have $P_1\circ \Gamma^{-1}(\widehat{\mathbf{O}}_S)=\Xi^{-1}(\mathbf{O}_S)$. Then it follows that
$$\Xi^{-1}(\mathbf{O}_S)=P_1\circ \Gamma^{-1}(\widehat{\mathbf{O}}_S)=P_1\circ \Gamma^{-1} \circ P_2^{-1}(\mathbf{O}_S)$$
is open in $\widetilde{\mathbf{N}}_q$ since $P_1$ is an open mapping.  This means that $\Xi$ is continuous. 

On the other hand, let $\Lambda=\Lambda_2\circ \Lambda_1^{-1}$ be the homeomorphism from $\mathbf{N}_q$ to $\mathbf{S}_q$. We have the following commutative diagram. 
\begin{displaymath}
    \xymatrix{
        \mathbf{N}_q \ar[r]^{\Lambda} \ar[d]^{P_1} & \mathbf{S}_q \ar[l]  \ar[d]^{P_2} \\
        \widetilde{\mathbf{N}}_q \ar[r]^{\Xi} &        \widetilde{\mathbf{S}}_q  \ar[l]
    }
\end{displaymath}
Then, for any open subset $\mathbf{O}_N$ of $\widetilde{\mathbf{N}}_q$, we obtain that $P_1^{-1}(\mathbf{O}_N)$ is open in  $\mathbf{N}_q$ since $P_1$ is continuous. 
It follows that $\Lambda \circ P_1^{-1}(\mathbf{O}_N)$ is open in $\mathbf{S}_q$ since $\Lambda$ is a homeomorphism. 
Consequently, we see that 
$$\Xi(\mathbf{O}_N)=P_2\circ \Lambda \circ P_1^{-1}(\mathbf{O}_N)$$ is open in $\widetilde{\mathbf{S}}_q$ since $P_2$ is an open mapping. This proves that $\Xi^{-1}$ is continuous.  The lemma is proved. \end{proof}

\begin{proof}[Proof of Proposition \ref{abers-2}]
Proposition \ref{abers-2} follows from Proposition \ref{abers-1} and Lemma \ref{dengjia}.
\end{proof}

The following result is also needed in the proof of Theorem \ref{th-3}.
\begin{lemma}\label{well}
Let $\mu, \nu \in \mathcal{M}({\Delta}^{*})$. For each $t\in \mathbb{R}$, if $\mu$ is asymptotically equivalent to $\nu$, then $\beta_{f_{\mu}}(t)=\beta_{f_{\nu}}(t)$.
\end{lemma}

\begin{remark}
In particular, $\beta_{f_{\mu}}(t)=0$ for any  $t\in\mathbb{R}$ if $f_{\mu}$ is an asymptotically conformal mapping. Here we say $f_{\mu}$ is an {\em asymptotically conformal mapping} if $\mu$ is  asymptotically equivalent to $0$. From Lemma \ref{well}, we see that $$B_{b}(t)=\sup\limits_{[\mu]_{AT}\in AT}\beta_{f_{\mu}}(t)$$ for each $t\in \mathbb{R}$.
\end{remark}

\begin{proof}[Proof of Lemma \ref{well}]
The case $t=0$ is obvious, we only consider $t\neq 0$.  When $\beta_{f_{\mu}}(t)=\gamma >0$.
Since $\mu$ is asymptotically equivalent to $\nu$, then we  know from Proposition \ref{ase-2} that
\begin{eqnarray}
\|N_{f_{\mu}}(z)-N_{f_{\nu}}(z)\|_{E_1}=|(N_{\mathbf{h}}\circ f_{\mu}(z))\cdot f'_{\mu}(z)|(1-|z|^2)\rightarrow 0,\,\, {\text as}\,\,\, |z| \rightarrow 1^{-}. \nonumber
\end{eqnarray}
Here $\mathbf{h}=f_{\nu}\circ f_{\mu}^{-1}$.  It follows that, for any $\varepsilon\in (0, \gamma)$, there is an $r\in (0,1)$ such that  $$\sup\limits_{|z|\in (r, 1)}|N_{f_{\mu}}(z)-N_{f_{\nu}}(z)|(1-|z|^2) <\varepsilon/|t|.$$
By (1) of Lemma \ref{key-lemma}, we have $|\beta_{f_{\mu}}(t)-\beta_{f_{\nu}}(t)|\leq \varepsilon.$ This implies that $\beta_{f_{\mu}}(t)=\beta_{f_{\nu}}(t)$.
When $\beta_{f_{\mu}}(t)=0$. For any $\varepsilon>0$, by using (2) of Lemma \ref{key-lemma}, we can similarly prove that $\beta_{f_{\nu}}(t)\leq \varepsilon.$ This means that $\beta_{f_{\nu}}(t)=0$.
The proof of Lemma \ref{well} is finished.
\end{proof}

We next present the proof of Theorem \ref{th-3}.
\begin{proof}[Proof of Theorem \ref{th-3}]
Lemma \ref{well} tells us that $I_{AT}$ is well-defined. In view of Proposition \ref{abers-2}, it suffices to prove that, for each $t\neq0$, the mapping
$$\Theta: [N_{f_{\mu}}]_{E_1}  \mapsto \beta_{f_{\mu}}(t),\, \mu\in \mathcal{M}(\Delta^{*})$$ is continuous on $\widetilde{\mathbf{N}}_q$.

For given $\mu \in \mathcal{M}(\Delta^{*})$.  When $\beta_{f_{\mu}}(t)=\gamma>0,$ for any $\varepsilon\in (0, \gamma)$, if some $N_{f_{\nu}}$ satisfies that
$$\|[N_{f_{\mu}}]_{E_1}-[N_{f_{\nu}}]_{E_1}\| < \varepsilon/|t|.$$
Then we know that there is a $\phi \in E_{1,0}$ such that
$$|N_{f_{\mu}}(z)-N_{f_{\nu}}(z)+\phi(z)|(1-|z|^2)< \varepsilon/|t|. $$
Consequently, we see that there is an $r\in (0,1)$ such that
$$\sup\limits_{|z|\in (r,1)}|N_{f_{\mu}}(z)-N_{f_{\nu}}(z)|(1-|z|^2)< \varepsilon/|t|.$$
It follows from (1) of Lemma \ref{key-lemma} again that $|\beta_{f_{\nu}}(t)-\beta_{f_{\mu}}(t)|\leq \varepsilon.$ This means that $\Theta$ is continuous at $[N_{f_{\mu}}]_{E_1}$ in $\widetilde{\mathbf{N}}_q$.  When $\beta_{f_{\mu}}(t)=0$, for any $\varepsilon>0$, we can similarly prove that $\beta_{f_{\nu}}(t)\leq \varepsilon$ by (2) of Lemma \ref{key-lemma} and so that $\Theta$ is continuous at $[N_{f_{\mu}}]_{E_1}$ in this case. This proves Theorem \ref{th-3}.
\end{proof}
\begin{remark}
Note that $d_{AT}([\mu]_{AT}, [\nu]_{AT}) \leq d_{T}([\mu]_T, [\nu]_T)$ for any $\mu, \nu \in \mathcal{M}(\Delta^{*})$, we see that the statement that $I_{AT}$ is continuous on $AT$ also implies Theorem \ref{th-1}.
\end{remark}

\section{\bf Final results and remarks}
\subsection{A final main theorem} Let $A$ be a subset of $\widehat{\mathbb{C}}$.  A  {\em holomorphic motion} of $A$ is a map $\mathbf{H}: \Delta \times A \rightarrow \widehat{\mathbb{C}}$  such that:

$\bullet$ for each fixed $z\in A$, the map $\lambda \mapsto  \mathbf{H}(\lambda, z)$ is holomorphic in $\Delta$;

$\bullet$  for each fixed $\lambda \in \Delta$, the map $z \mapsto  \mathbf{H}(\lambda, z)$ is injective in $\Delta$;

$\bullet$  for all $z\in A$, we have $\mathbf{H}(0, z)=z$.

Holomorphic motions were introduced in \cite{MSS} by Ma\~n\'e, Sad and Sullivan, who proved the $\lambda$-lemma. Slodkowski later established in \cite{Slo} the extended $\lambda$-lemma, which confirmed a conjecture of Sullivan and Thurston \cite{ST}. The theory of holomorphic motions have many applications in complex analysis and holomorphic dynamics, see \cite{AIM}. Holomorphic motions are closely related to quasiconformal mappings. It is known that
\begin{proposition}\label{last}
Let $\mu\in \mathcal{M}(\Delta^{*})$ and let $k=\|\mu\|_{\infty}$. Then there exists a (canonical) holomorphic motion $\mathbf{H}: \Delta \times \widehat{\mathbb{C}}\rightarrow \widehat{\mathbb{C}}$ such that $\mathbf{H}(k, z)=f_{\mu}(z)$. Moreover, when $k=\|\mu\|_{\infty}>0$, for each fixed $\lambda\in \Delta$, $\mathbf{H}(\lambda,z)$ is a quasiconformal mapping from $\widehat{\mathbb{C}}$ to itself with
$\mu_{\mathbf{H}(\lambda, z)|_{\Delta}}=0,\,\, z\in \Delta$ and $\mu_{\mathbf{H}(\lambda, z)|_{\Delta^{*}}}=\lambda \cdot \frac{\mu(z)}{\|\mu\|_{\infty}},\, z\in \Delta^{*}.$
\end{proposition}

\begin{remark}\label{frem}  For any $\mu\in \mathcal{M}(\Delta^{*})$ with $|\mu\|_{\infty}=k$.  We see from Proposition \ref{last} that there is a holomorphic motion $\mathbf{H}: \Delta \times \widehat{\mathbb{C}} \rightarrow \widehat{\mathbb{C}}$ such that $\mathbf{H}(k, z)=f_{\mu}(z)$. We denote $\mathbf{H}|_{\Delta\times \Delta}:={\mathbf{H}}_{\lambda}(z)$. For fixed $\lambda\in \Delta$, ${\mathbf{H}}_{\lambda}(z)$ is univalent in $\Delta$ and ${\mathbf{H}}_{k}(z)=f_{\mu}(z)$ for any $z\in \Delta$. Then we know that, for fixed $z\in \Delta$,
$$\lambda  \mapsto  [{\mathbf{H}}'_{\lambda}(z)]^t$$
is holomorphic in $\Delta$. For fixed $r\in (0,1)$, by \cite[Theorem 2.4.8]{Ran}, we have
$$\lambda  \mapsto \int_{0}^{2\pi}|{\mathbf{H}}'_{\lambda}(re^{i\theta})|^t d\theta$$
is subharmonic in $\Delta$. Note that
$$\beta_{{\mathbf{H}}_{\lambda}}(t)=\lim\limits_{\gamma\rightarrow 1^{-}} \sup\limits_{r\in (\gamma,1)}\frac{\log \int_{0}^{2\pi}|{\mathbf{H}}'_{\lambda}(re^{i\theta})|^td\theta}{|\log(1-r)|}.$$
Then, by the potential theory,  it is reasonable to guess that $\lambda \mapsto \beta_{{\mathbf{H}}_{\lambda}}(t)$ may satisfy the maximum modulus principle and so that $\beta_{f_{\mu}}(t)<B_b(t)$
for all $\mu \in \mathcal{M}(\Delta^{*})$ when $t\neq 0$. We will prove that this guess is true. 
\end{remark}
\begin{theorem}\label{main-f}
Let $t\neq 0$. For any $\mu\in \mathcal{M}(\Delta^{*})$, we have 
\begin{equation}\label{m-est}
\beta_{f_{\mu}}(t)<B_b(t).
\end{equation}
 \end{theorem}
 
To prove Theorem \ref{main-f}, we need the following lemma. 
\begin{lemma}\label{m-l}
Let $\mu\in \mathcal{M}(\Delta^{*})$. {\bf (1)} If $\beta_{f_{\mu}}(t_1)>0$ for some $t_1>0$,  then $\beta_{f_{\mu}}(t)$ is strictly increasing on $[t_1, +\infty)$. 
 {\bf (2)} If $\beta_{f_{\mu}}(t_1)>0$ for some $t_1<0$,  then $\beta_{f_{\mu}}(t)$ is strictly decreasing on $(-\infty, t_1]$. 
\end{lemma}

\begin{proof}
When $t_1>0$, let $\beta=\beta_{f_{\mu}}(t_1)>0$.  Let $\varepsilon<\beta$ be a positive number, which will be fixed later. Then we see from the definition of integral means spectrum of $f_{\mu}$ that  
there is a sequence $\{r_n\}_{n=1}^{\infty}$  with $r_n<1$ and $r_n \rightarrow 1$ as $n\rightarrow \infty$ and such that
\begin{equation}
\int_{0}^{2\pi} |f_{\mu}'(r_ne^{i\theta})|^{t_1} d\theta > \frac{1}{(1-r_n)^{\beta-\varepsilon}}.\nonumber
\end{equation}
We let

$$\mathcal{A}_n:=\int_{0}^{2\pi} |f_{\mu}'(r_ne^{i\theta})|^{t_1} d\theta, \, \,\mathcal{D}_n:= \frac{1}{(1-r_n)^{\beta-\varepsilon}},\,\, n\in \mathbb{N},$$
and define 

$${\mathcal{E}}_n:=\{\theta:  |f_{\mu}'(r_ne^{i\theta})|^{t_1}> \frac{\mathcal{D}_n}{2\pi}, \, \theta\in [0, 2\pi)\}, $$
$$\mathcal{F}_n:=\{\theta:  |f_{\mu}'(r_ne^{i\theta})|^{t_1}\leq \frac{\mathcal{D}_n}{2\pi},\, \theta\in [0,2\pi)\}.$$
It is obvious that $\mathcal{E}_n \bigcup \mathcal{F}_n=[0, 2\pi)$ and $\mathcal{E}_n \bigcap \mathcal{F}_n=\emptyset.$ We denote 

$$I_{\mathcal{E}}=\int_{\mathcal{E}_n} |f_{\mu}'(r_ne^{i\theta})|^{t_1} d\theta, \,\, I_{\mathcal{F}}=\int_{\mathcal{F}_n} |f_{\mu}'(r_ne^{i\theta})|^{t_1} d\theta.$$

Case 1.  If $I_{\mathcal{E}} \geq \frac{1}{2}\mathcal{A}_n$, then for $\Delta t>0$,  we have 
\begin{eqnarray}
\int_{0}^{2\pi}|f_{\mu}'(r_ne^{i\theta})|^{t_1+\Delta t} d\theta & \geq & \int_{\mathcal{E}_n}|f_{\mu}'(r_ne^{i\theta})|^{t_1+\Delta t} d\theta  \nonumber \\
& \geq & \Big(\frac{\mathcal{D}_n}{2\pi}\Big)^{\frac{\Delta t}{t_1}} \int_{\mathcal{E}_n}|f_{\mu}'(r_ne^{i\theta})|^{t_1} d\theta \nonumber \\
&\geq &\frac{1}{2} \Big(\frac{\mathcal{D}_n}{2\pi}\Big)^{\frac{\Delta t}{t_1}}\mathcal{A}_n \nonumber \\
&\geq  & \frac{1}{2}[2\pi]^{-\frac{\Delta t}{t_1}} [\mathcal{D}_n]^{1+\frac{\Delta t}{t_1}}=:C_1(t_1, \Delta t)[\mathcal{D}_n]^{1+\frac{\Delta t}{t_1}}. \nonumber
\end{eqnarray}
Consequently, we obtain that 
\begin{eqnarray}\label{i}
\int_{0}^{2\pi}|f_{\mu}'(r_ne^{i\theta})|^{t_1+\Delta t} d\theta&\geq  & \frac{ C_1(t_1, \Delta t)}{(1-r_n)^{(\beta-\varepsilon)(1+\frac{\Delta t}{t_1})}}. 
\end{eqnarray}

Case 2. If $I_{\mathcal{F}} \geq \frac{1}{2}\mathcal{A}_n$, we set 
$$\mathcal{G}_n:=\{\theta:  |f_{\mu}'(r_ne^{i\theta})|^{t_1}\leq \frac{1}{16}\frac{\mathcal{D}_n}{2\pi},\, \theta\in [0, 2\pi)\},$$
$${\mathcal{H}_n}:=\{\theta:  \frac{1}{16}\frac{\mathcal{D}_n}{2\pi}<|f_{\mu}'(r_ne^{i\theta})|^{t_1}\leq \frac{\mathcal{D}_n}{2\pi},\,  \theta\in [0, 2\pi)\}.$$
We easily see that $\mathcal{G}_n \bigcup \mathcal{H}_n=\mathcal{F}_n$ and $\mathcal{G}_n \bigcap \mathcal{H}_n=\emptyset.$ Then we have
\begin{eqnarray}
\int_{\mathcal{H}_n}|f_{\mu}'(r_ne^{i\theta})|^{t_1} d\theta&=&\int_{\mathcal{F}_n}|f_{\mu}'(r_ne^{i\theta})|^{t_1} d\theta-\int_{\mathcal{G}_n}|f_{\mu}'(r_ne^{i\theta})|^{t_1} d\theta \nonumber \\
&\geq & \frac{1}{2}\mathcal{A}_n- 2\pi \frac{1}{16}\frac{\mathcal{D}_n}{2\pi}\geq  \frac{7}{16}\mathcal{D}_n. \nonumber
\end{eqnarray}
Hence, we have 
\begin{eqnarray}
\int_{0}^{2\pi}|f_{\mu}'(r_ne^{i\theta})|^{t_1+\Delta t} d\theta &\geq &\int_{\mathcal{H}_n}|f_{\mu}'(r_ne^{i\theta})|^{t_1+\Delta t} d\theta \nonumber \\
&\geq & \Big(\frac{\mathcal{D}_n}{32\pi}\Big)^{\frac{\Delta t}{t_1}}\int_{\mathcal{H}_n}|f_{\mu}'(r_ne^{i\theta})|^{t_1} d\theta \nonumber \\
&\geq & \frac{7}{16}[32\pi]^{-\frac{\Delta t}{t_1}}[\mathcal{D}_n]^{1+\frac{\Delta t}{t_1}}:=C_2(t_1, \Delta t)[\mathcal{D}_n]^{1+\frac{\Delta t}{t_1}}. \nonumber
\end{eqnarray}
Consequently, we obtain that 
\begin{eqnarray}\label{ii}
\int_{0}^{2\pi}|f_{\mu}'(r_ne^{i\theta})|^{t_1+\Delta t} d\theta&\geq  & \frac{ C_2(t_1, \Delta t)}{(1-r_n)^{(\beta-\varepsilon)(1+\frac{\Delta t}{t_1})}}. 
\end{eqnarray}

Now, for any $\Delta t>0$, taking $$\varepsilon=\frac{1}{2}\frac{\beta \Delta t}{t_1+\Delta t}\in (0, \beta),$$ we see that 
$$(\beta-\varepsilon)(1+\frac{\Delta t}{t_1})=\beta+\frac{\beta\Delta t}{2t_1}.$$
Thus, it follows from (\ref{i}) and (\ref{ii}) that 
\begin{eqnarray}\label{iii}
\int_{0}^{2\pi}|f_{\mu}'(r_ne^{i\theta})|^{t_1+\Delta t} d\theta&\geq  & \frac{ C_2(t_1, \Delta t)}{(1-r_n)^{\beta+\frac{\beta \Delta t}{2t_1}}}, \nonumber 
\end{eqnarray}
since $C_2(t_1, \Delta t)\leq C_1(t_1, \Delta t).$ This implies that 
$$\beta_{f_{\mu}}(t_1+\Delta t)\geq \beta+\frac{\beta \Delta t}{3t_1}>\beta=\beta_{f_{\mu}}(t_1).$$
This proves the part {\bf(1)} of the lemma.  

When $t_1<0, \beta_{f_{\mu}}(t_1)>0$, replacing $f'_{\mu}$ by $\frac{1}{f'_{\mu}}$ in the above arguments, we can similarly show that the part {\bf (2)} also holds.  The proof of Lemma \ref{m-l} is finished. 
\end{proof}

We now present the proof of Theorem \ref{main-f}. 

\begin{proof}[Proof of Theorem \ref{main-f}]
First, it is easy to see that (\ref{m-est}) obviously holds if $\beta_{f_{\mu}}(t)=0$.  We consider the case when $t>0, \beta_{f_{\mu}}(t)>0$. 
We define a locally univalent function $\mathbf{h}$ on the domain $\Omega:=f_{\mu}(\Delta)$ by 
\begin{equation}
\mathbf{h}(w)=\int_{0}^{w} [f'_{\mu}(\mathbf{g}(\zeta))]^{\varepsilon}d\zeta,\,\, w\in \Omega. \nonumber
\end{equation}
Here $\mathbf{g}=f_{\mu}^{-1}$, $\varepsilon>0$ is a small number and $[f'_{\mu}(z)]^{\varepsilon}=\exp(\varepsilon \log f'_{\mu}(z)), z \in \Delta$. Then 
\begin{equation}\label{f-d}
\mathbf{h}'\circ f_{\mu}(z)=[f_{\mu}'(z)]^{\varepsilon}, z\in \Delta,
\end{equation}
and 
\begin{equation}\label{s-d}
\mathbf{h}''\circ f_{\mu}(z) \cdot f'_{\mu}(z)=\varepsilon[f_{\mu}'(z)]^{\varepsilon-1}f''_{\mu}(z), z\in \Delta. 
\end{equation}
Consequently, we obtain from (\ref{f-d}) and (\ref{s-d}) that
\begin{eqnarray}
|N_{\mathbf{h}}(w)|\rho_{\Omega}^{-1}(w)&=&\Big|\frac{\mathbf{h}''(w)}{\mathbf{h}'(w)} \Big|\frac{1-|\mathbf{g}(w)|^2}{|\mathbf{g}'(w)|}\nonumber \\
&=&\Big|\frac{\mathbf{h}''\circ f_{\mu}(z)}{\mathbf{h}'\circ f_{\mu}(z)}\Big| \cdot |f'_{\mu}(z)|(1-|z|^2)\nonumber \\
&=& \varepsilon |N_{f_{\mu}}(z)|\rho_{\Delta}^{-1}(z) \nonumber \\
&\leq& 6\varepsilon.  \nonumber
\end{eqnarray} 
Hence, from \cite{Bec} or \cite{AG1}, \cite{Zhur}, we see that  $\mathbf{h}$ is bounded univalent in $\Omega$ and admits a quasiconformal extension (still denoted by $\mathbf{h}$)  to $\widehat{\mathbb{C}}$ when $\varepsilon$ small enough.  Now, we let  $\varepsilon$ be small enough so that $\mathbf{h}$ has a quasiconformal extension to $\widehat{\mathbb{C}}$. Then, let $\mathbf{F}(z)=\mathbf{h}\circ f_{\mu}(z), z\in \widehat{\mathbb{C}}$ and let $\mathbb{F}=\mathbf{F}|_{\Delta}$, we see that $\mathbb{F}$ belongs to $\mathcal{S}_q$, and for any $r\in (0, 1)$, $\theta \in [0,2\pi)$, it holds that
\begin{eqnarray}
|\mathbb{F}'(re^{i\theta})|^{t}=
|\mathbf{h}'\circ f_{\mu}(re^{i\theta})|^{t}|f'_{\mu}(re^{i\theta})|^{t}=|f'_{\mu}(re^{i\theta})|^{t+t\varepsilon}. \nonumber 
\end{eqnarray}
Thus we have $\beta_{\mathbb{F}}(t)$=$\beta_{f_{\mu}}(t+t\varepsilon)$. It follows from the part {\bf (1)} of Lemma \ref{m-l} that $\beta_{\mathbb{F}}(t)>\beta_{f_{\mu}}(t)$, which implies that $\beta_{f_{\mu}}(t)<B_b(t)$.  

We next consider the case when $t<0, \beta_{f_{\mu}}(t)>0$. Repeating the above arguments and by using the part {\bf (2)} of Lemma \ref{m-l}, we can show that $\beta_{f_{\mu}}(t)<B_b(t)$ is true in this case. The proof of Theorem \ref{main-f} is finished.   
\end{proof}

It is easy to see from Theorem \ref{main-f} that
\begin{corollary}\label{cor}
Let $t \neq 0$. Then we have  $\beta_f(t)<B_b(t)$ for any $f\in \mathcal{S}_q$.
\end{corollary}
\begin{remark}
This corollary tells us that the extremal function for $B_b(t)$ can not be from the class $\mathcal{S}_q$ when $t\neq 0$. This partially answers part {\bf (3)} of Question \ref{ques}.
\end{remark}

\subsection{Remarks}

We continue to consider the class ${\mathcal{T}}$.  For any $\phi$ belonging to $\overline{{\mathcal{T}}}$, the closure of ${{\mathcal{T}}}$ in $E_1$, there is a unique univalent function $f_{\phi}(z)$ determined as in (\ref{exa}) with $f_{\phi}\in \mathcal{S}$ and such that
$\phi(z)=N_{f_{\phi}}(z).$ By checking carefully the proof of Theorem \ref{th-2}, we can prove that
 \begin{proposition}\label{Curve}
For each $t\in\mathbb{R}$, the IMS functional $I_{\overline{\mathcal{T}}}: \phi\mapsto \beta_{f_{\phi}}(t)$ is continuous on $\overline{\mathcal{T}}$.
  \end{proposition}
  
  \begin{remark}
 Although we know that $I_{\overline{\mathcal{T}}}$ is continuous on $\overline{\mathcal{T}}$, we can not conclude that  $I_{\overline{\mathcal{T}}}$ on $\overline{\mathcal{T}}$ attains a maximum, since $\overline{\mathcal{T}}$ is not compact in $E_1$.  On the other hand, we know that $\mathcal{S}$ is compact under the locally uniformly convergence topology. However, we find that the functional $I_{\mathcal{S}}: f \mapsto \beta_f(t), f\in \mathcal{S}$ is not continuous on $\mathcal{S}$ under this topology. An  easy example is $\kappa_r:=\kappa(rz), \, r\in (0,1)$, which is locally uniformly convergent to the Koebe function $\kappa$ as $r\rightarrow 1^{-}$, but $\beta_{\kappa_r}(t)=0$ for all $r\in (0,1)$ and $\beta_{\kappa}(t)=3t-1$ when $t>\frac{1}{3}$.
 \end{remark}

 \begin{remark} Corollary \ref{cor} tells us that, if the functional $I_{\overline{\mathcal{T}}}$  attains a maximum on $\overline{\mathcal{T}}$ for $t \neq 0$,  then there is an extremal function $f$ for $B_b(t)$ whose Pre-Schwarzian derivative $N_f$ lies in $\partial \mathcal{T}$. In fact, for some special cases, we know that $I_{\overline{\mathcal{T}}}$ attains a maximum at certain point in $\partial{\mathcal{T}}$.   For example,

 ($\mathbf{I}$) when $t\geq 2$, we know that $B_b(t)=t-1$ and $B_b(t)$ has an extremal function $\mathcal{F}:=-\log(1-z)$. We see that $N_{\mathcal{F}}$ belongs to $\partial \mathcal{T}$. Actually,  for $\gamma\in (0,1)$, let $f_{\gamma}(z):=[(1-z)^{1-\gamma}-1]/(\gamma-1).$ It is easy to see that $\lim\limits_{\gamma\rightarrow 1^{-}} \|N_{f_{\gamma}}-N_{\mathcal{F}}\|_{E_1}=0$ and we can check that $f_{a}(\mathbb{T})$ is a quasicircle for any
 $\gamma\in (0,1)$ and so that $N_{f_\gamma} \in \mathcal{T}$. Here we say a Jordan curve $\Gamma$ in $\mathbb{C}$ is a {\em quasicircle} if there is a constant $C(\Gamma)>0$ such that the diameter $l(z,w)$ of the smaller subarc $\wideparen{zw}$ of $\Gamma$ joining any two points $z$ and $w$ in $\Gamma$ satisfies that $l(z,w)\leq C(\Gamma)|z-w|$.
Hence $I_{\overline{\mathcal{T}}}$ attains a maximum at the point $N_{\mathcal{F}}$ in $\partial \mathcal{T}$ when $t\geq 2$.

 ($\mathbf{II}$)  From Theorem \ref{CM} and (\ref{est--0}), we know that $B_{b}(t)=|t|-1$ has an extremal function $\mathcal{G}:=-\frac{1}{2}[(1-z)^2-1]$ when $t\leq t_0$. Here $t_0$ is the same as in Theorem \ref{CM}.  We will show that $N_{\mathcal{G}}$ belongs to $\partial \mathcal{T}$. In fact, for $\gamma\in (0,1)$, let $g_{\gamma}(z):=[(1-z)^{1+\gamma}-1]/(-\gamma-1).$  We see that  $g_{\gamma}(\mathbb{T})$ is a quasicircle so that $N_{g_\gamma} \in \mathcal{T}$ for any
 $\gamma\in (0,1)$. On the other hand, it is easy to check that $\lim\limits_{\gamma\rightarrow 1^{-}} \|N_{g_{\gamma}}-N_{\mathcal{G}}\|_{E_1}=0.$ Thus, we see that
  $I_{\overline{\mathcal{T}}}$ attains a maximum at the point $N_{\mathcal{G}}$ on $\partial \mathcal{T}$ when $t\leq t_0$.  
  
However, we do not know
 \begin{question}\label{q-2}
Whether the IMS functional $I_{\overline{\mathcal{T}}}$ attains a maximum on $\overline{\mathcal{T}}$ for each $t\neq 0$?
  \end{question}
 \end{remark}
 
In view of the examples  ($\mathbf{I}$) and  ($\mathbf{II}$),  it is natural to raise the following 
\begin{conjecture}\label{con-1}
For each $t\in \mathbb{R}$, $B_{b}(t)$ has at least one extremal function whose Pre-Schwarzian derivative lies in $\partial \mathcal{T}$.
\end{conjecture}

\begin{remark}
From Proposition \ref{abers-2}, we can identify the universal asymptotic Teichm\"uller space $AT$ with $\widetilde{\mathbf{N}}_q$. Based on Conjecture \ref{con-1}, we propose the following   
\begin{conjecture}\label{con-2}
For each $t\in \mathbb{R}$, $B_{b}(t)$ has at least one extremal function $f$ such that $[N_f]_{E_1}$ lies in the boundary of the universal asymptotic Teichm\"uller space $\widetilde{\mathbf{N}}_q$ in $E_{1}/E_{1,0}$.
\end{conjecture}
\end{remark}

\begin{remark}We finally identify the universal Teichm\"uller space with $T_2=\mathbf{S}_q$. Hence we can see $\overline{T}_2=\overline{\mathbf{S}}_q$ as another model of the closure of the universal Teichm\"uller space. From \cite{AG2}, we know that the boundary $\partial\mathbf{S}_q$ of $\mathbf{S}_q$ is larger than that of $\mathbf{N}_q$.  If $\phi\in \mathbf{S}_q$, then there is a unique univalent function $f_{\phi}$ with $f_{\phi}\in \mathcal{S}_q^{\infty}$ and such that $\phi(z)=S_{f_{\phi}}(z).$  Hence, the IMS functional $I_{T_2}: \phi \mapsto \beta_{f_{\phi}}(t)$ is well-defined on $\mathbf{S}_q$.  From Proposition \ref{bers-2}, \ref{bers-1} and Theorem \ref{th-1}, we obtain that
\begin{corollary}
For each $t\in \mathbb{R}$, the IMS functional $I_{T_2}: \phi \mapsto \beta_{f_{\phi}}(t)$ is continuous on $\mathbf{S}_q$.
\end{corollary} 
 If $\phi\in \partial\mathbf{S}_q$, we know that there is a sequence $\{f_n(z)\}_{n=1}^{\infty}$, $f_{n}\in \mathcal{S}_q^{\infty}$ such that $\lim_{n\rightarrow\infty}\|S_{f_n}-\phi\|_{E_2}=0,$ and for each $z\in \Delta$, the sequence $\{f_n(z)\}_{n=1}^{\infty}$ converges. Then, take $f_\phi(z)=\lim_{n\rightarrow\infty}f_n(z), z\in \Delta$, we see that $f_{\phi}\in \mathcal{S}$ with $\phi(z)=S_{f_{\phi}}(z)$. In view of the normalization,  we believe that $f_{\phi}$ should be unique, here the statement  $f_{\phi}$ is unique means that, if there is another sequence $\{\widehat{f}_n(z)\}_{n=1}^{\infty}$, $\widehat{f}_{n}\in \mathcal{S}_q^{\infty}$ such that $\lim_{n\rightarrow\infty}\|S_{\widehat{f}_n}-\phi\|_{E_2}=0,$ and for each $z\in \Delta$, the sequence $\{\widehat{f}_n(z)\}_{n=1}^{\infty}$ converges, then, take $\widehat{f}_\phi(z)=\lim_{n\rightarrow\infty}\widehat{f}_n(z)$, we have $\widehat{f}_{\phi}(z)=f_{\phi}(z)$ for any $z\in \Delta$.  But we have not found a proof for this statement. We leave it as the following conjecture, which appears to be nontrivial.
 \begin{conjecture}\label{q-3}
 Let $\phi \in \overline{\mathbf{S}}_q$. Then $f_{\phi}$, taken as above, is unique so that the IMS functional $I_{\overline{T}_2}: \phi \mapsto \beta_{f_{\phi}}(t), \phi \in \overline{\mathbf{S}}_q$ is well-defined.
  \end{conjecture}
Furthermore, after resolving Conjecture \ref{q-3}, we can then consider the following
  \begin{question}\label{q-2}
Whether the IMS functional $I_{\overline{T}_2}: \phi \mapsto \beta_{f_{\phi}}(t)$ is continuous on $\overline{\mathbf{S}}_q$?
  \end{question}
 \end{remark} 
 
 \begin{spacing}{1.2} 

\end{spacing}
\end{document}